\documentclass[12pt]{amsart}
\usepackage[latin1]{inputenc}
\usepackage{color}
\usepackage{bbm}
\usepackage{amsmath,amssymb}
\usepackage{enumerate}
\usepackage{mathrsfs}
\usepackage{verbatim}

\newtheorem{thm}{Theorem}[section]
\newtheorem{cor}[thm]{Corollary}
\newtheorem{lem}[thm]{Lemma}
\newtheorem{prop}[thm]{Proposition}
\theoremstyle{definition}
\newtheorem{defin}[thm]{Definition}
\newtheorem{fact}[thm]{Fact}
\newtheorem{rem}[thm]{Remark}
\newtheorem{example}[thm]{Example}

\numberwithin{equation}{section}

\makeatletter
\@namedef{subjclassname@2020}{\textup{2020} Mathematics Subject Classification}
\makeatother

\makeatletter
\newcommand{\bZ}{\mathbb{Z}}

\newcommand{\bR}{\mathbb{R}}
\newcommand{\lin}{\operatorname{span}}
\newcommand{\supp}{\operatorname{supp}}

\newcommand{\ccE}{\mathscr{E}}

\allowdisplaybreaks

\begin{document}
\title{An algebraic characterization of B-Splines}
\author[A. Kamont]{Anna Kamont}
\address{ Institute of Mathematics, Polish Academy of Sciences, ul. Abrahama 18, 81-825 Sopot, Poland }
\email{ Anna.Kamont@impan.pl }

\author[M. Passenbrunner]{Markus Passenbrunner}
\address{Institute of Analysis, Johannes Kepler University Linz, Austria, 4040 Linz, Alten\-berger Strasse 69}
\email{markus.passenbrunner@jku.at}



\begin{abstract}
	B-splines of order $k$ can be viewed as a mapping $N$ taking a $(k+1)$-tuple of
	increasing real numbers $a_0 < \cdots < a_k$ and
	giving as a result a certain piecewise polynomial function. Looking at
	this mapping $N$ as a whole, basic properties of B-spline functions imply that it has
	the following algebraic properties: (1) $N(a_0,\ldots,a_k)$ has
	local support; (2) $N(a_0,\ldots,a_k)$ allows refinement, i.e. 
	for every $a\in \cup_{j=0}^{k-1} (a_j,a_{j+1})$ we have that if  
		$(\alpha_0,\ldots,\alpha_{k+1})$ is the increasing rearrangement
		of the points $\{a_0,\ldots,a_k,a\}$, the 'old' function
		$N(a_0,\ldots,a_k)$ is a linear combination of the 'new'
		functions $N(\alpha_0,\ldots,\alpha_k)$ and
		$N(\alpha_1,\ldots,\alpha_{k+1})$;
	(3) $N$ is translation and dilation invariant. It is easy to see that
	derivatives of $N(a_0,\ldots,a_k)$ satisfy properties (1)--(3) as well.
	
	In this paper we investigate if properties (1)--(3) are already sufficient to
	characterize B-splines and their derivatives. 
\end{abstract}

\maketitle

\section{Introduction}
For a positive integer $k$, let $a_0 < \cdots < a_k$ be an increasing sequence of real numbers. 
Recall the definition of B-spline functions.
The B-spline function $N(a_0, a_1 ; \cdot)$ of order~$1$ with respect to
the points $\{a_0,a_1\}$ is given by the characteristic
function of the interval $(a_0,a_1]$ and inductively, the B-spline function of
order $k$ with respect to the points $\{a_0,\ldots,a_k\}$ is 
\begin{equation}\label{eq:spline_rec}
	N(a_0,\ldots, a_k ; t) = \frac{ t - a_0 }{a_{k-1} - a_0}
	N(a_0,\ldots,a_{k-1} ; t) + \frac{ a_k - t}{a_k  -a_1}
	N(a_1,\ldots,a_{k};t).
\end{equation}
Then, the functions $N(a_0,\ldots,a_k) = N(a_0,\ldots,a_k ; \cdot)$ are
piecewise polynomials of order $k$ (degree $k-1$). 
If the distances $a_{j+1} - a_{j}$ between the points $\{a_0,\ldots,a_k\}$ are all equal,
$N(a_0,\ldots,a_k)$ is a \emph{cardinal} B-spline.

B-spline functions $N(a_0,\ldots,a_k)$ have the following algebraic properties.
\begin{enumerate}
	\item the support of $N(a_0,\ldots,a_k)$ equals $[a_0,a_k]$. 
	\item for every $a\in \cup_{j=0}^{k-1} (a_j,a_{j+1})$ we have that if  
		$(\alpha_0,\ldots,\alpha_{k+1})$ is the increasing rearrangement
		of the points $\{a_0,\ldots,a_k,a\}$,
		\[
			N(a_0,\ldots,a_k) \in \lin\{
				N(\alpha_0,\ldots,\alpha_k),
				N(\alpha_1,\ldots,\alpha_{k+1}) \}.
		\]

	\item Translation invariance 
		\[
			N(a_0 + \tau, \cdots, a_k+\tau)  \in \lin\{ N(a_0,\ldots,a_k;
			\cdot-\tau)\},\qquad \tau\in\mathbb R.
		\]
	\item Dilation invariance
		\[
			N(\delta a_0,\delta a_1,\ldots,\delta a_k) \in \lin  
			\{N(a_0,a_1,\ldots,a_k; \cdot/\delta)\},\qquad \delta >
			0.
		\]
\end{enumerate}
In fact, for the setting $N$ of B-splines as in \eqref{eq:spline_rec}, we actually
have equality in items (3) and (4). For more information on B-spline functions, see
e.g. \cite{Schumaker2007}.
We put $D_k = \{ (a_0,\ldots,a_k) \in \mathbb R^{k+1} : a_0 < \cdots < a_k\}$
and view the B-splines of order $k$ as a mapping $N : D_k \to \mathscr E'(\mathbb R)$, 
where $\mathscr E'(U)$ denotes the space of compactly supported distributions on
the open subset $U$ of $\mathbb R$, which is the dual of the space $\mathscr
E(U)$ of infinitely differentiable functions on $U$. 
Observe that the above conditions (1)--(4) make sense in $\mathscr E'(\mathbb
R)$ as well.
It is easily seen that distributional derivatives of
$N(a_0,\ldots,a_k)$ also satisfy (1)--(4).
In addition to properties (1)--(4), $N:D_k\to\mathscr E'(\mathbb R)$ is also a continuous mapping.
For notational purposes, define the B-spline $N(a)$ of order
$0$ at $a\in\mathbb R$ to be the Dirac distribution $\delta_a$ at point~$a$.

In this article, we show the distinctiveness of B-spline functions
\eqref{eq:spline_rec} by proving 
that the algebraic properties (1)--(4) of a continuous, non-vanishing mapping
$F: D_k \to
\mathscr E'(\mathbb R)$ in fact
characterize B-spline functions, i.e. for all $a_0 < \cdots < a_k$ we have that 
$F(a_0,\ldots,a_k)$ already is (a non-zero constant multiple of) the B-spline function
corresponding to the points $(a_0,\ldots,a_k)$ or some distributional derivative
thereof.
In other words, B-splines are characterized by their localization property (1),
their possibility of refinement (2) and their invariance (3), (4) with respect
to the two group
structures $+$ and $\cdot$ on $\mathbb R$ respectively.

In fact, we even show generalizations of this characterization in the following
sense. 
Firstly, we need not assume continuity of the mapping $F:D_k \to\mathscr
E'(\mathbb R)$ if we restrict ourselves to the dense subset $P_k\subset
D_k$ of tuples with rational coordinates. In particular, our main Theorems \ref{thm:main_large} and
\ref{thm:main_small} do not require continuity of $F$ as an assumption.
Therefore our characterization of B-splines by conditions (1)--(4) is
purely algebraic.
Secondly, instead of the support condition in  (1) we only assume its weaker version that
the support of $F(a_0,\ldots,a_k)$ is a \emph{subset} of $[a_0,a_k]$. 

Our motivation for considering this problem is twofold. Firstly, if we restrict
conditions (1)--(4) to equidistant partitions, mappings satisfying those
conditions are related to totally refinable distributions (see
Definition~\ref{defn:totally}), for which there are 
existing results about their form, cf. \cite{refinable_approx,  refinable_east, refinable_anal}.
Our results can be seen as an extension of those investigations to arbitrary
partitions.
Secondly, there are results 
about convergence properties of orthogonal projections onto an increasing
sequence of spline spaces  with respect to Lebesgue measure 
that are true for any sequence of partitions, cf. \cite{shadrin,uncond_franklin, ae,
uncond}.
Those results require as a crucial ingredient
that conditions (1) and (2) are true for general partitions. If we additionally assume that our
mappings $F$ are invariant with respect to the same group structures than
the Lebesgue measure, we arrive at mappings $F$ satisfying (1)--(4) and the 
results in this article show that in some sense, B-splines are uniquely
determined by those conditions.

The article is organized as follows. In Section~2 we describe our main results.
In Section~3 we investigate what can be said about a mapping $F$ satisfying
(1)--(4) by purely looking at equidistant partitions and using results about
totally refinable distributions.  In Section~4 we look at
general partitions and prove that $F(a_0,\ldots,a_k)$ is of the desired form for
a certain subset $P_{k,e}\subset P_k$. This statement can be extended to all of $P_k$ if the supports of
$F(a_0,\ldots,a_k)$ are sufficiently close to $[a_0,a_k]$, which will be seen in
Section~5.

\section{Main results}

Let $k$ be a non-negative integer and let 
\[
	P_k=\{(a_0,a_1,\ldots,a_k)\in\mathbb
	Q^{k+1} : a_0 <
a_1<\cdots <a_k\}.
\]
Let $F:P_k\to \mathscr E'(\mathbb R)$ be a mapping; for $(a_0,\ldots,a_k)\in P_k$
and $t\in\mathbb R$ we abbreviate $F(a_0,\ldots,a_k;t)= F(a_0,\ldots,a_k)(t)$. 
We assume that the mapping $F$ satisfies the following conditions for all
$(a_0,\ldots,a_k)\in P_k$:

\begin{enumerate}[(i)]
	\item \emph{Support condition:}
		\[
			\supp F(a_0,\ldots,a_k)\subseteq [a_0,a_k],
		\]
		where $\supp T$ denotes the \emph{support} of a distribution $T$ 
		given by the complement of the largest open subset of $\mathbb
		R$ on which $T$ vanishes.
	\item \emph{Span condition: } if $k\geq 1$, for all $a\in
		\cup_{j=0}^{k-1} (a_j,a_{j+1})\cap \mathbb Q$  we have that if 
		$(\alpha_0,\ldots,\alpha_{k+1})$ is the increasing rearrangement
		of the points $\{a_0,\ldots,a_k,a\}$,
		\[
			F(a_0,\ldots,a_k) \in \lin
			\{F(\alpha_0,\ldots,\alpha_k),
				F(\alpha_1,\ldots,\alpha_{k+1}) \}.
		\]

	\item \emph{Translation invariance: } 
		\[
			F(a_0 + \tau, \ldots, a_k+\tau) \in \lin \{F(a_0,\ldots,a_k;
			\cdot-\tau)\},\qquad \tau\in\mathbb Q.
		\]
	\item \emph{Dilation invariance: }
		\[
			F(\delta a_0,\delta a_1,\ldots,\delta a_k) \in \lin 
			\{F(a_0,a_1,\ldots,a_k; \cdot/\delta)\},\qquad \text{for
			positive }\delta\in \mathbb Q.
		\]
\end{enumerate}

In the case $k=0$, condition (ii) is empty and the mapping $F: P_0 \to
\ccE'(\mathbb R)$ is uniquely determined by conditions (i),(iii),(iv) if we know
$F(0)$. Thus, from now on we assume that $k\geq 1$.

Fixing $k_\ell \leq k_r$
with $k_\ell \geq 0$ and $k_r\leq k$ and setting
$F(a_0,\ldots,a_k)$ to be the B-spline corresponding to the points
$(a_{k_\ell},\ldots, a_{k_r})$ for all $(a_0,\ldots,a_k)\in P_k$ satisfies
(i)--(iv).

Moreover, it is easy to see that given a solution $F$ of (i)--(iv), 
derivatives
of $F$ are also solutions of (i)--(iv), i.e. 
we have the following proposition.

\begin{prop}\label{prop:Derivative}
	Suppose that $F:P_k\to \mathscr E'(\mathbb R)$ satisfies (i)-(iv).

	Then, the setting
	\[
		(a_0,\ldots,a_k)\mapsto DF(a_0,\ldots,a_k), \qquad
		(a_0,\ldots,a_k)\in P_k
	\]
	also satisfies (i)--(iv), where $D$ denotes the distributional
	derivative operator.
\end{prop}

We will show that B-Splines and their derivatives are the only solutions to
(i)--(iv) in the sense of the theorems below.

Assume that $F: P_k\to \mathscr E'(\mathbb R)$ is a mapping satisfying (i)--(iv)
as above that is not identically zero. Consider the distribution $F(0,1,\ldots,k-1,k)$ for the equidistant
splitting. We know that $F(0,\ldots,k)$ is not identically zero, since if it
were and $n_0 < \cdots < n_k$ would be an arbitrary increasing sequence of integers, we
infer by the span condition
\[
	F(n_0,\ldots,n_k) = \sum_{j=n_0}^{n_k-k} \lambda_j F(j,\ldots,j+k)
\]
for some coefficients $(\lambda_j)$ and therefore, $F(n_0,\ldots,n_k)$ would be
identically zero for all choices of points in $P_k$ (by translation and dilation invariance of
$F$).

Thus, we define the numbers $k_\ell \leq  k_r$ by setting $k_\ell$ to be the
largest integer $i$ so that $F(0,\ldots,k)$ is identically zero on $(-\infty,i)$
and we set $k_r$ to be the smallest integer so that $F(0,\ldots,k)$ is
identically zero on $(i,\infty)$.
Those numbers $k_\ell$ and $k_r$ satisfy $k_\ell\geq 0$ and $k_r\leq k$ and 
will be crucial in our investigations.

If the support of $F(0,\ldots,k)$ is sufficiently wide, we have the
following theorem.
\begin{thm}[Wide Support]
	\label{thm:main_large}
	Suppose that $F$, not vanishing identically, satisfies conditions (i)-(iv)
	and 	
	we have $\max(k_\ell, k-k_r) \leq 2$.

	Then, for some non-negative integer $n$ and all $(a_0,\ldots,a_k) \in
	P_k$,
	$F(a_0,\ldots,a_k)$ is a constant, non-zero multiple of the $n$th distributional
	derivative of the B-spline function of order $k_r - k_\ell$ with respect
	to the points $a_{k_\ell}, a_{k_\ell+1} ,\ldots ,a_{k_r}$.
\end{thm}

This result is a consequence of the more precise result of
Theorem~\ref{thm:main_large_refined}.

For narrower supports of $F(a_0,\ldots,a_k)$, we cannot expect such a general
result, as we will see in Example \ref{ex:counter}. Nevertheless, consider the
point class $Z$ consisting of all $(a_0,\ldots,a_k)\in\mathbb Z^{k+1}$ so that
$a_j-a_{j-1} = 1$ for all $j = 1,\ldots, k_\ell$ and $j=k_r
+1,\ldots, k$  and let 
\begin{equation}\label{eq:defPe}
	P_{k,e} := \{ \lambda (a_0,\ldots,a_k) : \lambda\in \mathbb
Q, (a_0,\ldots,a_k)\in Z\}\subset P_k.
\end{equation}

Then we have the following theorem.

\begin{thm}[Narrow Support]
	\label{thm:main_small}
Suppose that $F$, not vanishing identically, satisfies conditions (i)-(iv). 

	Then, for some non-negative integer $n$ and all $(a_0,\ldots,a_k)\in
	P_{k,e}$, $F(a_0,\ldots,a_k)$ is a
	constant, non-zero multiple of the $n$th distributional derivative of the B-spline 
	function of order $k_r - k_\ell$ with
	respect to the points $a_{k_\ell}, a_{k_\ell + 1},\ldots,a_{k_r}$.
\end{thm}

\begin{rem}
	By subdividing for instance equidistantly and using the assertion of
Theorem \ref{thm:main_small} and the span condition we obtain for all 
	$(a_0,\ldots,a_k)\in P_k$ that $F(a_0,\ldots,a_k)$ is the $n$th derivative of
	some spline (but not necessarily a B-spline) of order $k_r-k_\ell$ with the same parameter $n$ as in
	Theorem~\ref{thm:main_small}.
\end{rem}

\begin{example}\label{ex:counter}
	We now give an explicit example for $F$ satisfying (i)--(iv) with $k=4$, $k_\ell
	= 3$, $k_r = 4$ that does not allow for the stronger conclusion in
	Theorem \ref{thm:main_large}. 

	Set 
	\begin{align*}
		A &:= \{ (a_0,\ldots,a_4) \in P_4 : (a_2 - a_1)/(a_1 - a_0) =2,
			(a_3 - a_2)/(a_1 - a_0)\leq
		1\}
	\end{align*}
	and define
	\begin{align*}
		F(a_0,\ldots,a_4) &= \delta_{a_2} - \delta_{a_4},\qquad
		(a_0,\ldots,a_4) \in A,	\\
		F(a_0,\ldots,a_4) &= \delta_{a_3} - \delta_{a_4},\qquad
		(a_0,\ldots,a_4) \in P_4\setminus A.
	\end{align*}
	It is apparent that this setting of $F$ satisfies conditions (i), (iii)
	and (iv) and an easy calculation shows that this $F$ also satisfies
	the span condition (ii). For this calculation and a generalization of
	this counterexample, we refer to Appendix \ref{ss.ext.counter}.
	For $(a_0,\ldots,a_4)\in A$, this mapping $F$ does not satisfy the
	conclusion of Theorem \ref{thm:main_large}.
	This choice of $F$ is not continuous on $P_4$ but note that continuity of
	$F$ is not part of the assumptions of both Theorem \ref{thm:main_large}
	and Theorem \ref{thm:main_small}.
\end{example}

\section{Equidistant partitions}\label{s.equi.new}

While working in the equidistant case, we conveniently define $E_k\subset P_k$ by 
\[
	E_k=\{(a_0,a_1,\ldots,a_k)\in P_k : a_{j+1}-a_j = a_{j}-a_{j-1}\text{
	for all $j=1,\ldots,k-1$}\}.
\]
Our aim in this section is to find out the form of $F: E_k \to \ccE'(\bR)$.

\subsection{Some facts about compactly supported distributions.}\label{equiv.s1}
For non-negative integers $j$, we define the function $u_j(x) := x^j$ for
$x\in\mathbb R$.
\begin{defin}\label{def:moment}
Let $m \geq 1$ be an integer. We say that $\phi \in \ccE'(\bR)$ satisfies moment conditions of
order $m$ if $\phi(u_j) = 0$ for all $j=0, \ldots, m-1$.
\end{defin}

Let us note that  if $\phi \not\equiv 0$, then there is $j \geq 0$ such that
$\phi (u_j) \neq 0$. Indeed, the Paley-Wiener theorem states that if $\phi \in
\ccE'(\bR)$, then $\hat{\phi}$ is an entire function, given by the formula
$\hat{\phi} (z) =  \phi_x ( e^{-i x z}).$
Expanding $\hat{\phi}$ into power series we find
$\hat{\phi}(z) = \sum_{j=0}^\infty a_j z^j,$
where
$a_j =  \frac{1}{j!} D^j \hat{\phi}(0) = \frac{ 1}{j!}   \phi ( (- i)^j u_j).$
As $\hat{\phi} \not \equiv 0$, it follows that there is $j \geq 0$ such that
$\phi (u_j) \neq 0$.

\begin{fact}\label{fact.2}
Let $m \geq 1$ be an integer, and let $\phi \in \ccE'(\bR)$ satisfy moment conditions of order $m$. Then there is a unique $\psi \in \ccE'(\bR)$ such that 
$\phi = D^{m} \psi$. In addition, if $\phi(u_m) \neq 0$, then $\hat{\psi}(0) \neq 0$.
\end{fact}
\begin{proof}
Fix $f \in \ccE(\bR)$. Let $g \in \ccE(\bR)$ be such that $D^{m} g= f$.
The defining condition for $\psi$ states that we must have
$$
\phi(g) = D^m \psi (g) = (-1)^m \psi (D^m g) = (-1)^m \psi(f).
$$
Therefore, $\psi$ - if it exists - must be defined by the formula
$$
\psi(f) = (-1)^{m} \phi(g).
$$
Note that $\phi$ satisfies moment conditions of order $m$, and if $D^m g_1 = D^m g_2$, then $g_1 - g_2$ is a polynomial of order $m$.
It follows that $\psi$ is well-defined.
To see continuity of $\psi$ on $\ccE(\bR)$, fix $x_0 \in \bR$, and consider $g_{x_0}$ such that $D^m g_{x_0} = f$ on $\bR$ and $D^j g_{x_0} (x_0) = 0$
for $j=0, \ldots, m-1$. Note that the mapping $f \to g_{x_0}$ is linear and continuous on $\ccE(\bR)$, and we can take $\psi(f) = (-1)^m \phi(g_{x_0})$.

In addition, if $\phi(u_m) \neq 0$, then $\hat{\psi}(0) = \psi( u_0) = (-1)^m
\phi(u_m/m!) \neq 0$.
\end{proof}

For later convenience, if $\phi$ satisfies moment conditions of order $m$, let
$D^{-m} \phi$ be the distribution $\psi$ given by Fact \ref{fact.2}.

\subsection{Totally refinable  distributions.}\label{equiv.s2}

In the sequel, we refer to results from \cite{refinable_approx} on the 
characterization of totally refinable distributions.

Let us recall the definition of a totally refinable distribution, see e.g.
\cite{refinable_approx}. Please note some difference in the comparison with the
definition of $p$-refinable and totally refinable distribution in
\cite{refinable_approx} -- Definition \ref{defn:totally} does not require $\hat{\phi}(0) \neq 0$.
\begin{defin}\label{defn:totally} Let $p \geq 2$, $p \in \bZ$. A compactly
	supported distribution $\phi$, $\phi \not \equiv 0$, is called {\em
	$p$-refinable} if there exists a sequence of coefficients
$\{ c_j, j \in \bZ\}$, with only finitely many non-zero terms such that
\begin{equation}\label{eq:refin}
\phi  = \sum_{j \in \bZ} c_j \phi(p \cdot - j).
\end{equation}
A compactly supported distribution $\phi$ is called \emph{totally refinable} if it is $p$-refinable for all $p \in \bZ$, $p \geq 2$.
\end{defin}

Theorem 3 of \cite{refinable_approx} implies the following representation of  totally
refinable distributions:
\begin{thm}\label{th.rep.1}
Let $\phi$ be a totally refinable distribution with $\hat{\phi}(0)=1$. Then,
there is a non-negative integer $\rho$ and a sequence of coefficients
$\{\alpha_j , j \in \bZ\}$ with only finitely many non-zero terms, such that
$$\phi = \sum_{j \in \bZ} \alpha_j N(j, \ldots, j+\rho),$$
where for $\rho=0$  we denote $N(j) = \delta_j$.
\end{thm}

However, there is the following:
\begin{prop}\label{prop:deriv}
Let $\phi \in \ccE'(\bR)$ be totally refinable. Then there exist an integer $m \geq
0$ and 
$\psi \in  \ccE'(\bR)$, which is totally refinable and  $\hat{\psi}(0) \neq 0$,
such that $\phi = D^m \psi$.
\end{prop}
\begin{proof} If $\hat{\phi}(0) \neq 0$, then it is enough to take $m=0$ and $\psi = \phi$.

Assume that $\phi$ satisfies moment conditions of order $m \geq 1$. Then $D^{-m} \phi \in \ccE'(\bR)$ is well defined (cf. Fact \ref{fact.2}).
To see that $D^{-m}\phi$ is totally refinable, observe the following:
\begin{itemize}
\item Let $\lambda, \eta \in \bR$, $\lambda \neq 0$. Then $\phi( \lambda \cdot - \eta)$ satisfies moment condition of order $m$, and 
$ D^{-m} \big( \phi( \lambda \cdot - \eta) \big) = \lambda^{-m} ( D^{-m} \phi ) (\lambda \cdot - \eta)$.
\end{itemize}
Fix $p \geq 2$, and let $\phi$ satisfy \eqref{eq:refin}. Then we have
\begin{eqnarray*}
 D^{-m} \phi & = & D^{-m} \Big( \sum_{j \in \bZ} c_j  \phi(p \cdot - j) \Big)
\\ & = & \sum_{j \in \bZ} c_j D^{-m}\big( \phi(p \cdot - j) \big) 
=  \sum_{j \in \bZ} c_j p^{-m} ( D^{-m} \phi )(p \cdot - j).
\end{eqnarray*}
It follows that $D^{-m} \phi$ is totally refinable. 

Finally, if we choose $m \geq 1$ so that  $\phi$ satisfies moment condition of order $m$, but not of order $m+1$, then
we have $\widehat{D^{-m} \phi}(0) \neq 0$, see Fact \ref{fact.2}.
\end{proof}

Combining Theorem \ref{th.rep.1} and Proposition \ref{prop:deriv}, we get the following:

\begin{cor}\label{cor:refin}
Let $\phi \in \ccE'(\bR)$ be totally refinable. 
Then there are two integers  $m,\rho\geq 0$ and a sequence of coefficients
$\{\alpha_j , j \in \bZ\}$ with only finitely many non-zero terms, such that
$$\phi = \sum_{j \in \bZ} \alpha_j D^mN(j, \ldots, j+\rho).$$
\end{cor}

For further reference, we make the following observation:
\begin{fact}\label{fact.support}
Fix  integers $m,\rho \geq 0$,  and a sequence of coefficients $\{c_j , j_{\min}
\leq j \leq j_{\max}\}$, with $c_{j_{\min}} \neq 0$, $c_{j_{\max}} \neq 0$.
For an increasing sequence $(a_j)$ of rational numbers,
consider $\psi = \sum_{j = j_{\min}}^{j_{\max}} c_j D^m N(a_j, \ldots,
a_{j+\rho})$.
Then
$ \min \supp \psi = a_{j_{\min}}$ and $ \max \supp \psi = a_{j_{\max} + \rho}$.
\end{fact}

\subsubsection{Correspondence between  totally refinable distributions $\phi$ and mappings $K :
E_k \to \ccE'(\bR)$.}\label{sss.v2}
Let us consider $K : E_k \to \ccE'(\bR)$. Conditions (i), (iii) and (iv) make
sense in this setting, but we need to replace the span condition (ii) by its equidistant version: 
\begin{itemize}
\item[(ii.e)] For $k\geq 1$,  fix $(a_0, \ldots, a_{{k}}) \in  E_{{k}}$ and $p \geq 2$, $p \in \bZ$. 
Let $b_0, \ldots, b_{p {k}}$ be a $p$-refinement of $(a_0, \ldots, a_{{k}})$, i.e. a   sequence of points such that $b_{pj} = a_j$ for
all $j=0, \ldots, {k}$, and $(b_s, \ldots, b_{s+ {k}} ) \in E_{{k}}$ for all $s = 0, \ldots , (p-1){k}$.
Then
$$
K(a_0, \ldots, a_{{k}}) \in \lin\{ K(b_s, \ldots, b_{s+ {k}} ) : s = 0, \ldots , (p-1){k} \}.
$$ 
\end{itemize}
Note that for the choice $k=0$, condition (ii.e) is empty.

There is the following  correspondence between $K$ satisfying (i), (ii.e),
(iii), (iv) and $\phi$ which is totally refinable.
Clearly, if $K: E_k \to \ccE'(\bR)$ satisfies (i), (ii.e), (iii), (iv), then
$K(0,\ldots, k)$ is totally refinable.

On the other hand, let $\phi$ be totally refinable. 
Let
$$
k_{\min} = \min \supp \phi, \quad k_{\max} = \max \supp \phi.
$$
Corollary \ref{cor:refin} implies -- via Fact \ref{fact.support} --  that
$k_{\min}, k_{\max } \in \bZ$.
Let $\kappa = k_{\max} - k_{\min}$. 
 Consider $K_\phi: E_\kappa \to \ccE'(\bR)$ given by
$$
K_\phi (k_{\min}, \ldots, k_{\max}) = \phi,
$$
and extend it to $E_\kappa$ by translations and dilations. That is, 
 if $(a_0, \ldots, a_\kappa) \in E_\kappa$, then there are $\alpha, \beta \in
 \mathbb Q$, $\alpha > 0$
such that
$(a_0, \ldots, a_\kappa) = \alpha \cdot( k_{\min}, \ldots, k_{\max} )  + \beta$, and $K_\phi$ is given by
 the formula 
\begin{equation}\label{eq:307}
K_\phi( \alpha k_{\min} + \beta, \ldots, \alpha k_{\max} + \beta )(\cdot) = \phi \Big( \frac{\cdot - \beta} {\alpha} \Big).
\end{equation}
Then $K_\phi$ satisfies (i), (ii.e), (iii), (iv).

\subsection{On the extension of $F:E_k \to \ccE'(\bR)$ to $F:P_k \to \ccE'(\bR)$.}\label{ss.extension}
Let $\phi$ be totally refinable, with $k_{\min} = \min \supp \phi$ and $k_{\max}
= \max \supp \phi$. Recall that $k_{\min}, k_{\max} \in \bZ$.
Fix $\tilde{k}_{\min}\leq k_{\min}$ and $\tilde{k}_{\max} \geq k_{\max}$, and set
$\tilde{k} = \tilde{k}_{\max} - \tilde{k}_{\min}$.
	Combining arguments from sections \ref{sss.v2} and \ref{sss.dim}, we see that the setting
$$
\tilde{K}_\phi(\tilde{k}_{\min}, \ldots ,\tilde{k}_{\max}) = \phi
$$
can be extended to $\tilde{K}_\phi:E_{\tilde{k}} \to \ccE'(\bR)$, which satisfies (i), (ii.e), (iii) and (iv).

The following questions seem natural:
\begin{itemize}
	\item[(Q.1)] Does $\tilde{K}_\phi$ admit an extension to $P_{\tilde{k}}$, which satisfies (i)-(iv)?
	\item[(Q.2)] If $\tilde{K}_\phi$ admits an extension to $P_{\tilde{k}}$, which satisfies (i)-(iv), is this extension unique?
\end{itemize}

Theorem \ref{thm:main_small} can be interpreted as follows.
Let $\tilde{K}_\phi:E_k \to \ccE'(\bR)$ satisfy (i), (ii.e), (iii), (iv).  
If $\tilde{K}_\phi$ admits an extension $\tilde{K}_\phi:P_k \to \ccE'(\bR)$ satisfying (i)-(iv), then there are $0 \leq k_\ell \leq k_r \leq k$ and $m \geq 0$
such that $\tilde{K}_\phi(0, \ldots, k) = D^m N(k_\ell, \ldots,
k_r)$. Clearly, this $\tilde{K}_\phi$ has at least one extension to $P_k$, which satisfies (i)-(iv): 
$\tilde{K}_\phi(a_0, \ldots, a_k) = D^m N(a_{k_\ell}, \ldots, a_{k_r})$.

We give examples of $\tilde{K}_\phi$ which are not of the form required by Theorem \ref{thm:main_small}. 
This leads to a negative answer to question (Q.1).
First, observe the  following:
\begin{itemize}
\item Each cardinal $B$-spline $N(0,\ldots, k)$ is totally refinable.
\item If $\phi$ is totally refinable and $p\geq 2$, $p \in \bZ$, then
	$\phi(\cdot/p)$ is also totally refinable.
\item If $\phi$ is totally refinable and $\mu \in \bZ$, then $\Delta_\mu \phi
	(\cdot)  = \phi(\cdot + \mu) - \phi(\cdot)$ is also totally refinable.
	If $\phi$ satisfies moment conditions of order $m$, then $\Delta_\mu
	\phi$ satisfies moment conditions of order $m+1$.
\item If $\phi$ is totally refinable and satisfies moment conditions of order $m$, $m \geq 1$, then
$D^{-m} \phi$  is totally refinable.
\end{itemize}
Applying these observations, it is possible to give examples of $F$   which satisfy (i), (ii.e), (iii) and (iv), but which are  not of the form required by
Theorem  \ref{thm:main_small}. 
\begin{itemize}
\item Fix $\rho \geq 1$ and $\nu \geq  2$. Put $k  = \rho \nu$, and define $F$
	on $E_k$ by translations and dilations of 
$F(0,1,\ldots, \rho \nu) = N(0, \nu, \ldots, \rho \nu)$. 
\item Fix $\rho \geq 1$, $\nu \geq 2$ and $\mu \in \bZ$.  Put $k  = \rho \nu + \mu$, and define $F$ on $E_k$ 
	by translations and dilations of 
$F(0,1,\ldots, \rho \nu + \mu) = N(\mu, \mu+ \nu, \ldots, \mu+ \rho \nu) - N(0, \nu, \ldots, \rho \nu)$. 
\item Fix $\rho \geq 1$, $\nu \geq  2$ and $\mu \in \bZ$.  Put $k  = \rho \nu + \mu$, and define $F$ on $E_k$ 
	by translations and dilations of 
$$F(0,1,\ldots, \rho \nu + \mu) (t) = \int_{-\infty}^t \big( N(\mu, \mu+ \nu, \ldots, \mu+ \rho \nu)(u) - N(0, \nu, \ldots, \rho \nu)(u) \big) du.$$
\end{itemize}

These (counter)-examples imply the negative answer to question (Q.1).

The answer to question (Q.2) is more complicated. Let $\tilde{K}_\phi(0, \ldots, k)$ be of the form required by Theorem \ref{thm:main_small}.
Then it admits an extension to $P_k$, which satisfies (i)-(iv). 
Theorem \ref{thm:main_small} identifies some 
subset $P_{k,e} = P_{k,e}(\tilde{K}_\phi)$, $E_k \subset P_{k,e}(\tilde{K}_\phi) \subset P_k$, such
that the extension of $\tilde{K}_\phi$ to $P_{k,e}(\tilde{K}_\phi)$ is unique.
Moreover,
Theorem \ref{thm:main_large} guarantees that if $\max(k_\ell , k - k_r) \leq 2$,
then the extension of $\tilde{K}_\phi$ to $P_k$ is unique. Now, we state the following:

\begin{prop}\label{prop.non.uniq}
Let the mapping $F:E_k \to \ccE'(\bR)$ be given by the formula  $F(a_0, \ldots, a_k) = D^m N(a_{k_\ell}, \ldots, a_{k_r})$.
If $k \geq 4$,  $ k_\ell < k_r$ and $\max(k_\ell, k-k_r) \geq 3$, then $F$ admits several extensions to $P_k$ satisfying (i)-(iv).
\end{prop}
The proof of Proposition \ref{prop.non.uniq} is based on an extension of Example
\ref{ex:counter}, which is given below in Appendix \ref{ss.ext.counter}. 
The proof of Proposition~\ref{prop.non.uniq} is completed in section \ref{ss.proof}.

\subsection{The form of $F:E_k \to \ccE'(\bR)$.}\label{equiv.s3}
 Let $F: P_k \to \ccE'(\bR)$ satisfy (i)-(iv).  Then $F(0, \ldots, k)$ is a
 compactly supported distribution, which is totally refinable.
Therefore, by Corollary~\ref{cor:refin},  there are $m \geq 0$, $\rho \geq 0$,
$j_{\min} \leq j_{\max}$ and coefficients $\{c_j, j_{\min} \leq j \leq j_{\max}\}$
with $c_{j_{\min}} \neq 0$ and $c_{j_{\max}} \neq 0$ such that
$$
F(0, \ldots , k) = \sum_{j= j_{\min}} ^{j_{\max}} c_j D^m N(j,\ldots, j+\rho).
$$
By Fact~\ref{fact.support},
$$
k_\ell = \min \supp F(0, \ldots, k) = j_{\min} , \quad k_r = \max \supp F(0,
\ldots, k) = j_{\max} + \rho.
$$
Setting $k' = k_r - k_\ell$, we see that
$$
\rho \leq \rho + j_{\max} - j_{\min} = k' \leq k.
$$
Now, define $G :E_{k'} \to \ccE'(\bR)$ by setting
$$
 G (k_\ell, \ldots, k_r) =  F(0, \ldots, k),
$$
and extend it to $E_{k'}$ by translations and  dilations, as in formula \eqref{eq:307}. Note that $G$ satisfies (i), (ii.e), (iii), (iv), cf. Section \ref{sss.v2}.
Observe that -- by definition of $G$ -- we have 
$$
k_\ell  =   \min \supp G(k_\ell, \ldots, k_r),
\quad
k_r  =  \max \supp G(k_\ell, \ldots, k_r).
$$

To summarize these considerations, we formulate the following:
\begin{prop}\label{prop:solequi}
 $G (0, \ldots, k') $ has one of the following forms:
		\begin{enumerate}
			\item 	There exist an integer $\rho$ with $1 \leq \rho \leq k'$ and coefficients  $\{\alpha_j, 0 \leq j \leq k' - \rho\}$, 
with $\alpha_0 \neq 0$ and $\alpha_{k' - \rho } \neq 0$ such that
$
G (0, \ldots, k') = \sum_{j=0}^{k' - \rho} \alpha_j N(j, \ldots, j + \rho).
$
			\item	There exist an integer $n \geq 0$ and coefficients $\{ h_j, 0 \leq j \leq k'\}$, with $h_0 \neq 0$ and $h_{k'} \neq 0$ such that 
$G(0,\ldots,k') = \sum_{j=0}^{k'}
				h_j D^{n} \delta_j$.
		\end{enumerate}
\end{prop}

Next, we collect more information about the coefficients $(h_j)$ appearing in
item (2) of Proposition~\ref{prop:solequi}.

\begin{prop}\label{prop:polyequi}
	Suppose that $F$ satisfies conditions (i)-(iv).

	If $G(0,\ldots,k') = F(-k_\ell,\ldots,k-k_\ell) = \sum_{j=0}^{k'} h_j
	D^n\delta_j$ for some non-negative integer $n$, the
	zeros of the polynomial $p(w) = \sum_{j=0}^{k'} h_j w^j$ are
	contained in the set 
	\[
	 \{ \exp(2\pi i \nu/j) : j=1,\ldots,k'; \nu=0,\ldots, j-1 \}.
	\]

Moreover, if $p\big(\!\exp(2\pi i \nu_0/j)\big) = 0$ for some $j \in \{1,\ldots,k'\}$ and
some $\nu_0\in \{0,\ldots,j-1\}$, we have $p\big(\!\exp(2\pi i \nu /j)\big)=0$ for all
$\nu\in \{0,\ldots,j-1\}$. 
\end{prop}
\begin{proof}
By the conditions (i), (ii.e), (iii), (iv) on $G$ for equidistant partitions, for all 
positive integers $r$, there exist coefficients $(\mu_m)$ satisfying
\begin{equation}\label{eq:1a}
	G(0,\ldots,k') = \sum_{m=0}^{k'(r-1)} \mu_m G\Big(
	\frac{m}{r},\ldots,\frac{m+k'}{r}\Big).
\end{equation}
By dilation invariance of $G$, for some coefficient $\lambda\neq 0$,
\[
	G(0,\ldots,k'/r;t) = \lambda G(0,\ldots,k';rt) = \lambda
	\sum_{j=0}^{k'} h_j
	(D^n\delta_j) (rt)
	= \frac{\lambda}{r^{n+1}} \sum_{j=0}^{k'} h_j D^n \delta_{j/r}(t).
\]
Therefore, translation invariance implies that equation \eqref{eq:1a} becomes
\begin{equation}\label{eq:1b}
	\sum_{j=0}^{k'} h_j D^n\delta_j = \sum_{m=0}^{k'(r-1)}
	\mu_m \sum_{j=0}^{k'} h_j D^n \delta_{(m+j)/r},
\end{equation}
for possibly different coefficients $(\mu_m)$ than the ones in \eqref{eq:1a}.
Applying to both sides the function $t\mapsto e^{t r z}$ for an arbitrary nonzero 
complex number $z$,
\[
	\sum_{j=0}^{k'} h_{j} e^{j r z } = \sum_{m=0}^{k'(r-1)}
	\mu_m e^{m z} \sum_{j=0}^{k'} h_j e^{j z}
\]
or equivalently
\begin{equation}\label{eq:pG}
	p(e^{rz}) = p(e^z)\sum_{m=0}^{k'(r-1)} \mu_m e^{m z}.
\end{equation}
This has consequences for the zeros of the polynomial $p$. First we observe that
$p(0)\neq 0$ by the fact that $h_0\neq 0$. Then, if $z$ is
such that $p(e^z)=0$, by equation \eqref{eq:pG} also $p(e^{rz})=0$. As the polynomial
$p$ has at most $k'$ zeros (including multiplicities), the zeros of $p$ must be contained in the set
\[
	 \{ \exp(2\pi i \nu/j) : j=1,\ldots,k'; \nu=0,\ldots, j-1 \},
\]
since otherwise $p(e^z)=0\implies p(e^{rz})=0$ would produce more than $k'$ zeros
of $p$.
Moreover, if $p(\exp(2\pi i \nu_0/j)) = 0$ for some $j \in \{1,\ldots,k'\}$ and
some $\nu_0\in \{0,\ldots,j-1\}$, we must have $p(\exp(2\pi i \nu /j))$ for all
$\nu\in \{0,\ldots,j-1\}$. Indeed, choosing $r$ that is relatively prime to
$j$, iterating the implication $p(e^z)=0\implies p(e^{rz})=0$ starting with
$z=2\pi i \nu_0/j$ and using the periodicity of the exponential function, we
obtain $p(\exp(2\pi i \nu/j))=0$ for all $\nu\in \{0,\ldots,j-1\}$.
\end{proof}

\section{General partitions and the proof of Theorem~\ref{thm:main_small}}\label{sec:generalpart}

First, we can use the information gained for equidistant partitions to prove a
result similar to Proposition \ref{prop:solequi}, but for general partitions
$(a_0,\ldots,a_k)\in P_k$.

	\begin{prop}\label{prop:sol}
	Suppose that $F$ satisfies conditions (i)-(iv).

	Then, there exists an integer $\rho$ with $\rho\leq k'$ so that for all 
		$(a_0,\ldots,a_k)\in P_k$, the following two
		statements are true:
		\begin{enumerate}
			\item If $1\leq \rho\leq k'$, $F(a_0,\ldots,a_k) =
				\sum_{j=0}^{k-\rho} \alpha_j
				N(a_j,\ldots,a_{j+\rho})$ for some coefficients
				$(\alpha_j)$ depending on $(a_0,\ldots,a_k)$. 
		\item	If $\rho\leq 0$, $F(a_0,\ldots,a_k) =
			\sum_{j=0}^{k}\alpha_j 
			D^{-\rho}
			\delta_{a_j}$
			for some coefficients $(\alpha_j)$ depending on
			$(a_0,\ldots,a_k)$.
	\end{enumerate}
\end{prop}
\begin{proof}
	By translation and dilation invariance of $F$ and the definition of the
	set $P_k$, we can assume that $(a_0,\ldots,a_k)=(n_0,n_1,\ldots,n_k)$ consists of an increasing
	sequence of integers with $n_0=0$ and by the span condition of $F$,
	\[
		F(n_0,n_1,\ldots,n_k) = \sum_{j=0}^{n_k - k} \mu_j F(j,\ldots,j+k) =
		\sum_{j=0}^{n_k - k} \mu_j G(j+k_\ell,\ldots,j+k_r)
	\]
for some coefficients $(\mu_j)$. The form of $G(j+k_\ell,\ldots,j+k_r)$ given by
Proposition \ref{prop:solequi} yields
that
$F(n_0,n_1,\ldots,n_k)=\sum_{j=n_0}^{n_k-\rho} \alpha_j N(j,\ldots,j+\rho)$ if $1\leq \rho\leq k'$ and
of the form $F(n_0,\ldots,n_k) = \sum_{j=n_0}^{n_k} \alpha_j D^{-\rho}\delta_j$ 
if $\rho\leq 0$. 
This implies that on each interval $(j,j+1)$, $F(n_0,\ldots,n_k)$ is a polynomial of order $\rho$ 
if $\rho\geq 1$ and vanishes identically if $\rho \leq 0$.
We need to check that the same statement is true for the intervals
$(n_i,n_{i+1})$.

Let  $i\in \{0,\ldots,k-1\}$
and let $p>\max(k,n_{i+1}-n_i)$ be a prime number so that 
$n_i-k(n_{i+1}-n_i)/p >
n_{i-1}$ if $i\geq 1$ and $n_{i+1} + k(n_{i+1}-n_i)/p < n_{i+2}$ if $i\leq
k-2$. We refine the partition $n_0 < n_1 < \cdots < n_k$ equidistantly in the vicinity of
$(n_i,n_{i+1})$ by inserting the points $n_i + j(n_{i+1} - n_i)/p$ for $j =  -k
, \ldots, p+k$.
By the span and support condition on $F$, the following identity holds on the interval
$(n_i,n_{i+1})$:
\begin{align*}
	F(n_0,\ldots,n_k) &= \sum_{j=-k+1}^{p-1} \lambda_j F\Big( n_i +
	\frac{j(n_{i+1}-n_i)}{p},\ldots,n_i +
	\frac{(j+k)(n_{i+1}-n_i)}{p}\Big) \\
	&=\sum_{j=-k+1}^{p-1} \lambda_j G\Big( n_i +
	\frac{(j+k_\ell)(n_{i+1}-n_i)}{p},\ldots,n_i +
	\frac{(j+k_r)(n_{i+1}-n_i)}{p}\Big).
\end{align*}
As above, by Proposition~\ref{prop:solequi},
$F(n_0,\ldots,n_k)$ is a polynomial of order  $\rho$ if $1\leq \rho\leq k'$ and
equals zero if $\rho\leq 0$ on the intervals $I_j=\big( n_i + j(n_{i+1} - n_i)/p, n_i +
(j+1)(n_{i+1}-n_i)/p\big)$. Since $p$ is prime, the
intervals $I_j$ for $j=0,\ldots,p-1$ and $(j,j+1)$ for $j =
n_i,\ldots,n_{i+1}-1$ cover the interval $(n_i,n_{i+1})$ and combining the above
assertions for $F$ yields that
$F(n_0,n_1,\ldots,n_k)=\sum_{j=0}^{k-\rho} \alpha_j N(n_j,\ldots,n_{j+\rho})$ if $1\leq \rho\leq k'$ and
of the form $F(n_0,\ldots,n_k) = \sum_{j=0}^{k} \alpha_j D^{-\rho}\delta_{n_j}$ 
if $\rho\leq 0$. 
\end{proof}

In order to gain even more information about the coefficients $(\alpha_j)$ above, we look at
the following analogue of the numbers $k_\ell,k_r$ for non-equidistant points.

\begin{defin}
	Let $F$ be a mapping satisfying (i)--(iv) and $(a_0,\ldots,a_k)\in P_k$. 

	If $F(a_0,\ldots,a_k)$ does not vanish identically,
choose $i,j$ so that $a_i = \min\supp F(a_0,\ldots,a_k)$ and $a_j = \max\supp
F(a_0,\ldots,a_k)$ and set 
\[
	\ell(a_0,\ldots,a_k) = i \qquad \text{and} \qquad  r(a_0,\ldots,a_k) =
	j.
\]
If $F(a_0,\ldots,a_k)$ vanishes identically, 
we set $\ell(a_0,\ldots,a_k)=\infty$ and $r(a_0,\ldots,a_k)=-\infty$.
\end{defin}

In view of Proposition~\ref{prop:sol} and Fact~\ref{fact.support}, the mappings 
$\ell$ and $r$ are well defined on $P_k$.

The subsequent results including Corollary~\ref{cor:Pe_ks} give information about the values of
$\ell(a_0,\ldots,a_k)$ and $r(a_0,\ldots,a_k)$ for certain points
$(a_0,\ldots,a_k)\in P_k$.

\begin{lem}\label{lem:contraction}
	Suppose that $F$ satisfies conditions (i)-(iv).
	
	For an increasing sequence of integers $a_0 <
	\cdots < a_k$  satisfying $a_i - a_{i-1} = 1$
	for $i\leq i_0$ we have
	\[
		\ell(a_0,\ldots,a_k) \geq \min\{i_0+1,k_\ell\}.
	\]

	Similarly, for an increasing sequence of integers $a_0 <
	\cdots < a_k$ satisfying $a_{i+1} - a_{i} = 1$
	for $i\geq i_0$, we have
	\[
		r(a_0,\ldots,a_k) \leq \max\{i_0-1,k_r\}.
	\]
\end{lem}
\begin{proof}
	We only prove the assertion for $\ell$ and without loss of generality, assume that $a_{0}=0$ and refine
	by inserting all integers as points.
	We obtain by the span condition of $F$ that
	\begin{equation}\label{eq:span10}
		F(a_0,\ldots,a_k) = \sum_{j\geq 0} \mu_j F(j,\ldots,j+k)
	\end{equation}
	for some coefficients $(\mu_j)$. The right hand side vanishes
	identically on
	$(-\infty,k_\ell)$
	and therefore the left hand side vanishes identically on the same
	interval which implies $\ell(a_0,\ldots,a_k) \geq \min\{ i_0 +1,
	k_\ell\}$.
\end{proof}

\begin{lem}\label{lem:lreq}
	Suppose that $F$ satisfies conditions (i)-(iv).

	Let $a_{k_\ell}
	< \cdots < a_{k}$ be an increasing sequence of integers and  
	set $a_0 < \cdots < a_{k_\ell-1} < a_{k_\ell}$
	so that this sequence is equidistant with $a_j - a_{j-1}
	= 1$ for every $j=1,\ldots,k_\ell$.
	Then,
	\begin{equation}\label{eq:l}
		\ell(a_0,\ldots,a_k) = k_\ell.
	\end{equation}
	
	Similarly, let $a_0' < \cdots < a_{k_r}'$ be an increasing sequence of
	integers and set $a_{k_r}' <a_{k_r+1}' < \cdots <
	a_k'$ so that this sequence  is equidistant with $a_j' - a_{j-1}'
	= 1$ for every $j=k_r+1,\ldots,k$. Then
	\begin{equation}\label{eq:r}
		r(a_0',\ldots,a_k') = k_r.
	\end{equation}
\end{lem}
\begin{proof}
	We only prove the assertion for $\ell$, the proof for $r$ is similar.
	Let $a_0 < \cdots < a_k$ be a sequence of integers with the above
	properties. 

	First, we treat the case $k_\ell=0$.
	Here, we let $0 =\alpha_0 < \cdots <\alpha_k=1$ be a rescaling of
	$(a_0,\ldots,a_k)$ and by the span and the support condition of $F$,
	\[
		F(0,\ldots,k) = \mu F(\alpha_0,\ldots,\alpha_k)\qquad\text{on
		$(-\infty,\alpha_1)$}.
	\]
	Since $F(0,\ldots,k)$ does not vanish identically on any neighborhood of
	zero,
	the same is true for the right
	hand side and in particular $\ell(a_0,\ldots,a_k) = 0$.

	Now we assume that $k_\ell \geq 1$.
	For $(b_0,\ldots,b_k)\in P_k$ so that $b_i$ is an integer for every 
	$i=0,\ldots,k$, we define the equidistant shifts
	$E$ to the left and $E^{-1}$ to the right:
	\[
		E(b_0,\ldots,b_k) := (b_1,\ldots,b_k,b_k+1),\qquad
		E^{-1}(b_0,\ldots,b_k):=(b_0 - 1,b_0,\ldots,b_{k-1}).
	\]
	Since $E^{-k}(b_0,\ldots,b_k)$ is equidistant by definition we know that
	$\ell(E^{-k}(b_0,\ldots,b_k))=k_\ell$. We will now show by induction on $i$
	that $\ell(E^i(a_0,\ldots,a_k)) = k_\ell$ for the given sequence of
	points $(a_0,\ldots,a_k)$ and all integers $i$ in the range $-k\leq i\leq  0$. Therefore,
	assume that $\ell(E^i(a_0,\ldots,a_k)) = k_\ell$ for $i < m\leq 0$ and
	we will now show this for $i=m$.

	To this end, 
	generate a special partition on the interval $[0,k]$. Let
	$(c_j)_{j=-M}^N$ be such that, with a sufficiently large integer $q$,
	\begin{enumerate}
		\item $0=c_{-M} < \cdots < c_{-1} < c_0$ is equidistant so that
			the subsequent points have distance $1/q$ to each other,
		\item $c_0 < \cdots < c_k$ is a translated rescaling of the points
			$(a_0,\ldots,a_k)$ with factor $1/q$,
		\item $c_k < \cdots < c_N=k$ is equidistant so that the
			subsequent points have distance $1/q$ to each other,
		\item \label{it:eq} $k_\ell = c_{m+k_\ell}$.
	\end{enumerate}
	Note that, for all integers $s$ in the range $-M
	\leq s \leq N-k$,  
	$(c_s,\ldots,c_{s+k})$ is a rescaling of
	$E^s(a_0,\ldots,a_k)$. Thus, by dilation and translation invariance of
	$F$ we get 
	$\ell(E^s(a_0,\ldots,a_k)) = \ell(c_s,\ldots,c_{s+k})$ for all such
	integers $s$.
	Therefore, by Lemma \ref{lem:contraction} and the induction hypothesis we know that
	\begin{equation}\label{eq:facts1}\left\{
		\begin{aligned}
			\ell(c_i,\ldots,c_{i+k}) & = k_\ell\qquad &\text{if 
			$i<m$}, \\
			\ell(c_i,\ldots,c_{i+k}) &\geq k_\ell\qquad&\text{ if
			$m\leq i\leq 1$}, \\
			\ell(c_i,\ldots,c_{i+k}) &\geq k_\ell-(i-1) \qquad&\text{
			if $i>1$.}
		\end{aligned}
		\right.
	\end{equation}
	
	By the span condition of $F$,
	\begin{equation}\label{eq:spancontract}
		F(0,\ldots,k) = \sum_{j=-M}^{N-k} \mu_j F(c_j,\ldots,c_{j+k})
	\end{equation}
	for some coefficients $(\mu_j)$.
	We now show by induction that the coefficient $\mu_j$
	 vanishes for $j< m$.
	 First we prove that $\mu_{-M}=0$ provided that $-M < m$.
	The left hand side of equation \eqref{eq:spancontract} is
	identically zero on $I=(-\infty,c_{-M+k_\ell+1})\subseteq (-\infty,
	c_{m+k_\ell}) = (-\infty, k_\ell)$ and also by
	\eqref{eq:facts1}, 
	$F(c_j,\ldots,c_{j+k})$ is identically zero on $I$ for
	$j>-M$. But
	$F(c_{-M},\ldots,c_{-M+k})$ is not identically zero on $I$  since
	$\ell(c_{-M},\ldots,c_{-M+k})=k_\ell$.
	This implies $\mu_{-M}=0$.

	As long as
	$c_{j+k_\ell+1} \leq c_{m+k_\ell}$, we can continue the same argument by
	induction on $j$ to deduce $\mu_j=0$ for $j<m$ and therefore, equation
	\eqref{eq:spancontract} becomes
	\[
		F(0,\ldots,k) = \sum_{j=m}^{N-k} \mu_j F(c_j,\ldots,c_{j+k}),
	\]
	which we now evaluate on $(-\infty, c_{m+k_\ell+1})$. 
	The facts \eqref{eq:facts1} and $m\leq 0$ now imply
	that $F(c_{m+s},\ldots,c_{m+s+k})$ vanishes identically on $(-\infty,c_{m+k_\ell+1})$ for
	all positive integers $s$ and therefore,
	\[
		F(0,\ldots,k) = \mu_m F(c_m,\ldots,c_{m+k})\qquad\text{on
		}(-\infty,c_{m+k_\ell+1}).
	\]
	By \eqref{it:eq}, $F(0,\ldots,k)$ does not
	vanish identically on $(-\infty,c_{m+k_\ell+1})$. This means in
	particular that $F(c_m,\ldots,c_{m+k})$ does not vanish identically on
	$(-\infty,c_{m+k_\ell+1})$, which implies that $\ell\big( E^m (
	a_0,\ldots,a_k ) \big) =\ell( c_m,\ldots,c_{m+k})  \leq k_\ell$.
	In combination with the inequality $\ell\big(
	E^{m}(a_0,\ldots,a_k)\big) = \ell(c_m,\ldots,c_{m+k}) \geq k_\ell$ from \eqref{eq:facts1}, 
	this shows the asserted
	equation $\ell(E^{m}(a_0,\ldots,a_k))  = k_\ell$, which completes the
	induction on $m$ and therefore, the proof of the lemma is finished as
	well.
\end{proof}

Recall the definition of $P_{k,e}$ in equation \eqref{eq:defPe} which consists of all rescalings
of integer tuples $(a_0,\ldots,a_k)$ so that $a_{j}-a_{j-1} = 1$ for every 
$j=1,\ldots,k_\ell$ and $j=k_{r}+1,\ldots,k$. Thus, in particular, the above
lemma implies the following corollary.
\begin{cor}\label{cor:Pe_ks}
	Suppose that $F$ satisfies conditions (i)--(iv).

	If $(a_0,\ldots,a_k)\in P_{k,e}$, we have $\ell(a_0,\ldots,a_k) = k_\ell$
	and $r(a_0,\ldots,a_k) = k_r$.
\end{cor}

\begin{lem}\label{lem:0}
	Suppose that $F$ satisfies conditions (i)-(iv) and suppose that for some
	non-negative integer $n$ and all 
	$(a_0,\ldots,a_k)\in P_k$, we have $F(a_0,\ldots,a_k) = \sum_{i=0}^k
	f_i D^n \delta_{a_i}$ for some coefficients $(f_i)$ depending
	on $(a_0,\ldots,a_k)$.

	Then, for all $(a_0,\ldots,a_k)\in P_{k,e}$ those coefficients $(f_i)$
	satisfy 
	\begin{align*}
		f_i \neq 0 \qquad \text{if $i\in\{k_\ell,\ldots,k_r\}$},
		\\
		f_i = 0 \qquad \text{if $i\notin\{k_\ell,\ldots,k_r\}$}.
	\end{align*}
\end{lem}
\begin{proof}
	Let $(a_0,\ldots,a_k)\in P_{k,e}$ be a sequence of integers.
	By Corollary~\ref{cor:Pe_ks}, we know that $f_i=0$ if $i\notin
	\{k_\ell,\ldots,k_r\}$ and $f_{k_\ell}\neq 0$ and $f_{k_r}\neq 0$.
	We argue by contradiction and assume that for some sequence of
	coefficients $(a_0,\ldots,a_k)\in P_{k,e}$
with $F(a_0,\ldots,a_k) = \sum_{i=0}^k f_i D^n\delta_{a_i}$,
we have $f_j = 0$ for some $j\in \{k_\ell+1,\ldots,k_r-1\}$.
	We now show that for all integers $i$, $F(E^i (a_0,\ldots,a_k))$ vanishes
	identically
	on a neighborhood of $a_j$, where we again use the equidistant shifts
	\[
		E(a_0,\ldots,a_k) := (a_1,\ldots,a_k,a_k+1),\qquad
		E^{-1}(a_0,\ldots,a_k):=(a_0 - 1,a_0,\ldots,a_{k-1}).
	\]
	Choose $a_{k+1} = a_k + 1$; by the span condition, there
	exist coefficients $\mu_1,\mu_2$ with
	\[
		F(a_0,\ldots, a_{j-1},a_{j+1},\ldots, a_{k+1}) = \mu_1
		F(a_0,\ldots,a_{k}) + \mu_2 F(a_1,a_2,\ldots,a_{k+1}).
	\]
	Evaluating this equation on a sufficiently small neighbourhood $U$ of
	$a_j$, by the form of $F$, the left hand
	side is identically zero on $U$ and by assumption
	$F(a_0,a_1,\ldots,a_k)$ is as well. This
	implies either $\mu_2=0$ or $F(a_1,\ldots,a_{k+1})\equiv 0$ on $U$.
	The case $\mu_2=0$ means
	\[
		F(a_0,\ldots, a_{j-1},a_{j+1},\ldots, a_{k+1}) = \mu_1
		F(a_0,a_1,\ldots,a_k),
	\]
	which is not possible since by Lemma \ref{lem:lreq}, the left hand side
	is not identically zero on each
	neighborhood of $a_{k_r+1}$
	but the right hand side is, provided this neighborhood is sufficiently
	small.
	Therefore, we are in the case that $F(a_{1},\ldots,a_{k+1}) =
	F(E(a_0,\ldots,a_k))=0$
	on $U$.
	Applying this argument by induction yields that
	$F(E^i(a_0,\ldots,a_k))=0$ on a neighborhood of $a_j$ for all positive
	integers $i$. 
	A symmetric argument yields the same assertion for negative integers $i$
	as well and in summary, for each integer $i$, $F(E^i(a_0,\ldots,a_k))=0$
	on a neighborhood of $a_j$.

	This yields a contradiction as follows. Consider for instance the
	following
	partition on the interval $[0,k]$. Let
	$(c_i)_{i=-M}^N$ be such that, with a sufficiently large integer $q$,
	\begin{enumerate}
		\item $0=c_{-M} < \cdots < c_{-1} < c_0$ is equidistant so that
			the subsequent points have distance $1/q$ to each other,
		\item $c_0 < \cdots < c_k$ is a translated rescaling of the points
			$(a_0,\ldots,a_k)$ with factor $1/q$,
		\item $c_k < \cdots < c_N=k$ is equidistant so that the
			subsequent points have distance $1/q$ to each other,
		\item $k_\ell = c_{j}$.
	\end{enumerate}
	The span condition of $F$ implies
	\[
		F(0,\ldots,k) = \sum_{i=-M}^{N-k} \mu_i F(c_i,\ldots,c_{i+k}).
	\]
	We evaluate this at $k_\ell=c_j$ and observe that by definition of
	$k_\ell$ and the
	form of $F$, $F(0,\ldots,k)$ is not identically zero on any neighborhood of
	$k_\ell$, whereas $(c_i,\ldots,c_{i+k})$ is a rescaling of
	$E^i(a_0,\ldots,a_k)$ and by dilation and translation invariance of $F$
	we therefore know that $F(c_i,\ldots,c_{i+k})$ vanishes identically in a
	neighborhood of $c_j=k_\ell$ for all integers~$i$. This is the desired
	contradiction.
\end{proof}

With the help of the results in this section, we are now able to
prove the exact form of the coefficients $(h_i)$ appearing in 
Proposition~\ref{prop:solequi} if $F$, for equidistant partitions, is of the form 
$ F(-k_\ell,\ldots,k-k_\ell) = \sum_{i=0}^{k'} h_i D^n \delta_i$.
\begin{thm}\label{thm:endform}
	Suppose that $F$ satisfies conditions (i)-(iv) and suppose that for some
	non-negative integer $n$ and all 
	$(a_0,\ldots,a_k)\in P_k$, we have $F(a_0,\ldots,a_k) = \sum_{i=0}^k
	f_{a_i} D^n \delta_{a_i}$ for some coefficients $(f_{a_i})$.

	Then,
	\[
F(-k_\ell,\ldots, k-k_\ell)=
\lambda\sum_{i=0}^{k'} (-1)^i {k'\choose i} D^n\delta_i
	\]
	for some nonzero coefficient $\lambda$.
\end{thm}

\begin{proof}
	\textsc{Part I: Preparation.}
	Let $0=n_0 < \cdots < n_{k'}$ be an arbitrary increasing sequence of
	integers and let $ n_{j} = j$ for $j<0$ and $n_j = n_{k'} + (j-k')$ for
	$j>k'$ be an equidistant extension.
	Then, we abbreviate, as in the equidistant case,
\begin{equation}\label{eq:delta1}
	G(n_0,\ldots,n_{k'}) :=
	F(n_{-k_\ell},\ldots,n_{k-k_\ell}) = \sum_{i=0}^{k'} f_{n_i} D^n
	\delta_{n_i},
\end{equation}
where the latter form follows from Lemma \ref{lem:0}. 
Lemma \ref{lem:lreq} implies that $G$ satisfies the span condition in the form
\begin{equation}\label{eq:delta2}
	G(n_0,\ldots,n_{k'}) = \sum_{j=0}^{n_{k'} - k'} \lambda_j G(j,\ldots,j+k')
\end{equation}
for some coefficients $(\lambda_j)$. 
Recall that the definition of $k_\ell$ and $k_r$ was such that $h_{0}\neq
0$ and $h_{k'}\neq 0$ and also, by Lemma \ref{lem:0}, $f_{n_0}\neq 0$ and
$f_{n_{k'}}\neq 0$. Without loss of generality, we assume the normalization $h_0
= f_{n_0} = 1$, which yields from \eqref{eq:delta2} that $\lambda_0 = 1$.

Equating \eqref{eq:delta1} and \eqref{eq:delta2} and applying the function
$t\mapsto e^{tz}$ with
$z\neq 0$ yields
 with the setting $w=e^z$ and $h_j = 0$
for $j\notin \{ 0,\ldots, k'\}$
\begin{equation}\label{eq:conv}
	\sum_{i=0}^{k'} f_{n_i} w^{n_i} = \sum_{j=0}^{n_{k'} - k'}
	\lambda_j w^j \sum_{i=0}^{k'} h_i w^i = \sum_{s=0}^{n_{k'}}
	w^s\sum_{j=0}^{n_{k'} - k'} \lambda_j h_{s-j}.
\end{equation}

\textsc{Part II.} Next, we show that for all choices of positive integers $m$
and all choices of
integers $s_1 < \cdots < s_m$ in the range $\{ 1,\ldots, k' + m-1\}$, the matrix
$(h_{s_i - j})_{i,j=1}^m$ is invertible.
Assume that this is not the case and let $m\geq 1$ be the smallest integer such
that there exist integers $s_1 < \cdots < s_m$ in the range $\{ 1,\ldots,k'+m-1
\}$ so that the matrix $(h_{s_i - j})_{i,j=1}^m$ is not invertible.
We already know that $m\geq 2$  since  by Lemma \ref{lem:0} we have $h_i\neq
0$ for all $i=0,\ldots,k'$.
Based on the sequence of points $s_1 < \cdots < s_m$, we define the following
special sequence $n_0 < \cdots < n_{k'}$:
Let $n_0=0$, $n_{k'} = k'+m-1$ and 
choose the sequence $n_1 <
\cdots < n_{k'-1}$ so that 
\begin{equation}\label{eq:def_s}
	\{0,1,\ldots, n_{k'}-1, n_{k'}\} \setminus \{ s_1,\ldots, s_{m-1} \} = \{
		n_0,n_1,\ldots,n_{k'-1},n_{k'}\}.
\end{equation}
By the minimality of $m$, the matrix $(h_{{s_i}-j})_{i,j=1}^{m-1}$ is
invertible. Therefore, by \eqref{eq:conv}, the coefficients
$\lambda_1,\ldots,\lambda_{m-1}$ are given uniquely by the equations (recall
that due to normalization, $\lambda_0 = 1$)
\[
	0 = \sum_{j=0}^{m-1} \lambda_j h_{s_i-j}= h_{s_i} +
	\sum_{j=1}^{m-1} \lambda_j h_{s_i-j},\qquad i=1,\ldots,m-1.
\]
Moreover, by assumption, the rows $v_1,\ldots,v_m$ of $(h_{s_i -
j})_{i,j=1}^{m}$ are linearily
dependent, i.e. there exist
coefficients $\mu_1,\ldots,\mu_m$, not all vanishing, with
$\sum_{i=1}^m \mu_i v_i = 0$. We know that $\mu_m\neq 0$ since otherwise
$\sum_{i=1}^{m-1} \mu_i v_i=0$ and thus in particular,  denoting $v_i'$ to be
the vector $v_i$ with the last entry deleted, $\sum_{i=1}^{m-1} \mu_i v_{i}' =
0$ which is not possible by the invertibility of the matrix $(v_i')_{i=1}^{m-1}
= (h_{{s_i}-j})_{i,j=1}^{m-1}$. Therefore, we have $v_m = \sum_{i=1}^{m-1}
\tau_i v_i$ for some coefficients $(\tau_i)$. But, using also \eqref{eq:conv},
this implies
\[
	f_{s_m} = \sum_{j=0}^{m-1} \lambda_{j} h_{s_m -j} = \sum_{i=1}^{m-1}
	\tau_i\sum_{j=0}^{m-1} \lambda_j h_{s_i - j} = 0 
\]
and since by \eqref{eq:def_s} 
we have $s_m= n_i$
for some $i = 0,\ldots,k'$, this implies $f_{n_i}= 0$.
But this is not possible by Lemma \ref{lem:0}. 
Therefore, the proof of the invertibility of the matrix $(h_{s_i -
j})_{i,j=1}^m$ for all positive integers $m$ and all choices of integers $s_1 < \cdots < s_m$ in the range
$\{1,\ldots, k' + m -1\}$ is finished.

\textsc{Part III. } Set $p(w) = \sum_{i=0}^{k'} h_i w^i$.
Let $r$ be the largest positive integer so that $p(w) = (1-w^r) v(w)$ for some
polynomial $v$. Since the degree of $p$ is $k'$, the number $r$ must be $\leq k'$ 
and by Proposition~\ref{prop:polyequi} we have $p(w) = (1-w) v(w)$ for some polynomial
$v$, which means that $r$ is well defined.
Next, define $n=r!$ and $\nu=k'-r+1$. 
We choose the special point sequence $0=n_0 < n_1 < \cdots < n_{k'-1}
< n_{k'} = n\nu$ to be $n_j = jn$ for $1\leq j\leq k'-r$ and $n_j$ arbitrary for
$j= k'-r+1,\ldots, k'-1$. Define the indices $s_1 < \cdots < s_{n\nu-k'} $
by the equation
\[
	\{s_1,\ldots, s_{n\nu-k'} \} = \{0,1,\ldots,n\nu-1,n\nu\} \setminus
	\{n_0,n_1,\ldots,n_{k'-1}, n_{k'}\}.
\]
By Part II of the proof, the matrix $\big( h_{s_i - j} \big)_{i,j=1}^{n\nu-k'}$ is
invertible, and thus
the coefficients 
$\lambda_1,\ldots,\lambda_{n\nu-k'}$ in \eqref{eq:delta2} and \eqref{eq:conv}
are given uniquely by the equations (recall
that $\lambda_0 = 1$)
\begin{equation}\label{eq:lambda}
	0 = \sum_{j=0}^{n\nu-k'} \lambda_j h_{s_i-j}=  h_{s_i} +
	\sum_{j=1}^{n\nu-k'} \lambda_j h_{s_i-j},\qquad i=1,\ldots,n\nu-k'.
\end{equation}
Define
\[
	q(w) := \frac{(1 - w^n )^\nu}{p(w)}.
\]
Due to Proposition \ref{prop:polyequi} about the structure of the roots of $p$ and the
definition of $r$ and $n=r!$, the function $q$ 
is a polynomial with degree $ n\nu - k'$. 
Thus, for some coefficients $(\lambda_j')_{j=0}^{n\nu - k'}$ with $\lambda_0'=1$ we have $q(w) =
\sum_{j=0}^{n\nu - k'} \lambda_j' w^j$.
The definition of the numbers $s_1 < \cdots < s_{n\nu-k'}$ implies that the
coefficient of $w^{s_i}$ of the polynomial $(p\cdot q)(w) = (1-w^n)^\nu$
vanishes for all $i=1,\ldots,n\nu-k'$. Thus the coefficients
$\lambda_1',\ldots,\lambda_{n\nu-k'}'$ are given by the same conditions
\eqref{eq:lambda} as the coefficients $\lambda_1,\ldots,\lambda_{n\nu - k'}$.
Therefore, we must have that $\lambda_j = \lambda_j'$ for every $j=0,\ldots,n\nu -
k'$, which also implies (cf. \eqref{eq:conv})
\[
	(1-w^n)^\nu = \sum_{i=0}^{k'} f_{n_i} w^{n_i}.
\]
Now, if $r\geq 2$, there exists an index $n_j$ that is not an integer multiple
of $n$ and $f_{n_j}$ is given by the coefficient of $w^{n_j}$ of the polynomial
$(p\cdot q)(w)=(1-w^n)^\nu$, which is zero. But this is not possible by Lemma
\ref{lem:0}. Therefore, $r=1$, and Proposition~\ref{prop:polyequi} about the structure
of the roots of $p$ implies
that 
\[
	p(w) = \sum_{i=0}^{k'} h_i w^i= (1-w)^{k'} = \sum_{i=0}^{k'} (-1)^i
	{k'\choose i}
	w^i,
\]
which completes the proof.
\end{proof}

Now we are in the position to prove one of our main results, viz.
Theorem~\ref{thm:main_small}.

\begin{proof}[Proof of Theorem \ref{thm:main_small}]
Let $F: P_k \to \mathscr E'(\mathbb R)$ be a mapping satisfying (i)--(iv). By
Proposition \ref{prop:Derivative} and Proposition \ref{prop:sol}, some
derivative $H_1$ of $F$ also satisfies (i)--(iv) and, for some non-negative
integer $n$ and all 
$(b_0,\ldots,b_k)\in P_k$ is of the form
\[
	H_1(b_0,\ldots,b_k) = \sum_{i=0}^{k} f_i D^n \delta_{b_i}
\]
for some coefficients $(f_i)$ depending on $(b_0,\ldots,b_k)$.
Then, we obtain from Theorem \ref{thm:endform}
that 
\[
	H_2(0,\ldots,k') := H_1(-k_\ell,\ldots,k-k_\ell)=
	\lambda\sum_{i=0}^{k'} (-1)^i {k'\choose i} D^n\delta_i
\]
for some non-zero coefficient $\lambda$. Therefore, $H_2(0,\ldots,k')$ is a multiple of
the $(n+k')$th
derivative of the cardinal B-spline function of order $k'$.
Moreover, by partial integration and the fact that B-spline functions have  
localized support, 
\begin{equation}\label{eq:equidistant}
	(H_2(0,\ldots,k';\cdot),u_j) = 0
\end{equation}
for all non-negative integer $j<n+k'$ and the function $u_j(x) = x^j$.
	Let $0=n_0 < \cdots < n_{k'}$ be an arbitrary increasing sequence of
	integers and let $ n_{j} = j$ for $j<0$ and $n_j = n_{k'} + (j-k')$ for
	$j>n_{k'}$ be an equidistant extension.
	As in equation \eqref{eq:delta2},
\[
	H_2(n_0,\ldots,n_{k'}) := H_1(n_{-k_\ell},\ldots,
	n_{k-k_\ell})=\sum_{j=0}^{n_{k'} - k'} \lambda_j H_2(j,\ldots,j+k')
\]
for some coefficients $(\lambda_j)$ and by \eqref{eq:equidistant},
\[
	(H_2(n_0,\ldots,n_{k'};\cdot), u_j)  = 0,\qquad j< n+k'.
\]
Those moment conditions determine the coefficients in the representation
$H_2(n_0,\ldots,n_{k'}) = \sum_{i=0}^{k'} f_i D^n \delta_{n_i}$ up to a
multiplicative constant. Indeed, the moment conditions are then equivalent to
the equations
\[
	\sum_{i=0}^{k'} f_i n_i^j=0,\qquad j=0,\ldots,k'-1.
\]
If we impose the additional normalization equation that $\sum_{i=0}^{k'} f_i
n_i^{k'} = 1$, the coefficients $f_0,\ldots,f_{k'}$ are given as the solution of a
linear system with a Vandermonde matrix that is invertible.
Therefore, also $H_1(n_{-k_\ell},\ldots,n_{k-k_\ell})$  is uniquely determined up to a
multiplicative constant and is thus a constant multiple of the $(n+k')$th
derivative of the B-spline function associated to the points
$(n_{0},\ldots,n_{k'})$. Hence, $F(n_{-k_\ell},\ldots,n_{k-k_\ell})$ differs from
some derivative of a B-spline function (corresponding to the points $n_0,\ldots
n_{k'}$) at most by a polynomial function $p$,
which, due to the involved support conditions must satisfy $p\equiv 0$. 
\end{proof}

Under the assumption that the functions $\ell$ and $r$ are constant on $P_k$,
we can improve the assertion of Theorem~\ref{thm:main_small} by a simple
argument. This will be discussed in the following corollary.

\begin{cor}\label{cor:lconstant}
Suppose that $F$ satisfies conditions (i)-(iv). 

If the functions $\ell$ and $r$ are constant on $P_k$, 
then, for some non-negative integer $n$ and all $(a_0,\ldots,a_k)\in
P_k$, $F(a_0,\ldots,a_k)$ is a
	constant, non-zero multiple of the $n$th distributional derivative of the B-spline 
	function of order $k_r - k_\ell$ with
	respect to the points $a_{k_\ell}, a_{k_\ell + 1},\ldots,a_{k_r}$.
\end{cor}
\begin{proof}
	Let $(a_0,\ldots,a_k)\in
	P_k$ be arbitrary. Then, there exists a refinement
	$(b_{-k_\ell},\ldots,b_{2k-k_r})$ of $(a_0,\ldots,a_k)$ with
	$a_j = b_j$ for $j\in \{k_\ell,\ldots,k_r\}$,
	$a_j = b_{j-k_\ell}$ for $j\in \{0,\ldots,k_\ell-1\}$ and 
	$a_j = b_{j+(k-k_r)}$ for $j\in \{k_r + 1,\ldots,k\}$
	so that $(b_{0},\ldots,b_{k})\in P_{k,e}$.
	By the span condition, we get
	\[
		F(a_0,\ldots,a_k) = \sum_{j=-k_\ell}^{k-k_r} \lambda_j
		F(b_j,\ldots,b_{j+k})
	\]
	for some coefficients $(\lambda_j)$. Since $\ell(a_0,\ldots,a_k) =
	\ell(b_j,\ldots,b_{j+k})=k_\ell$ for all $j$ by assumption we get that
	in a neighborhood of $b_{0}$, this equation becomes
	$	0 = \lambda_{-k_\ell} F(b_{-k_\ell},\ldots,b_{k-k_\ell})$,
	which implies $\lambda_{-k_\ell} = 0$. 
	Then, inductively, we deduce that $\lambda_j = 0$ for all
	$j\in \{-k_\ell,\ldots,-1\}$.
	Similarly, using $r$ instead of $\ell$, we obtain $\lambda_j = 0$ for
	$j\in \{1,\ldots,k-k_r\}$.
	Therefore, 
	\begin{equation}
		\label{eq:refine}
		F(a_0,\ldots,a_k) = \lambda_{0} F(b_{0},\ldots,
		b_{ k}).
	\end{equation}
	The left hand side does not vanish identically, which implies
	$\lambda_{0} \neq 0$. Moreover, by Theorem \ref{thm:main_small},
	since $(b_0,\ldots,b_k)\in P_{k,e}$, we obtain that 
	$F(b_{0},\ldots,b_{k})$ is a constant, non-zero multiple of
	the $n$th distributional derivative of the B-spline function $k_r-
	k_\ell$ with respect to the points $(b_{k_\ell},\ldots,b_{k_r}) =
	(a_{k_\ell},\ldots,a_{k_r})$, and thus, by \eqref{eq:refine}, so is $F(a_0,\ldots,a_k)$, which was to
	prove.
\end{proof}

\section{Small values of $k_\ell$ and large values of $k_r$ --- The proof of
	Theorem \ref{thm:main_large}}

In this section, we identify cases for $k_\ell$ and $k_r$ in which the
functions $\ell$ and $r$ are constant in order to deduce Theorem
\ref{thm:main_large}.
\begin{lem}\label{lem:ellzero}
	Suppose that $F$ satisfies conditions (i)--(iv).

	If $\ell(a_0,\ldots,a_k)=0$ for some choice of points $(a_0 , \ldots ,
	a_k)\in P_k$, then $\ell \equiv 0$ on $P_k$. 
	
	If $r(a_0,\ldots,a_k)=k$ for some choice of points $(a_0,
	\ldots,a_k)\in P_k$, then $r\equiv k$ on $P_k$.
\end{lem}
\begin{proof}
	We only prove the part about $\ell$, the proof for $r$ is similar.

	Fix $(a_0,\ldots,a_k)\in P_k$ with $\ell(a_0,\ldots,a_k)=0$.
	Let $b_0'=a_0 < b_1' < \cdots < b_k' = a_1$ be a rescaling of the arbitrary
	points $(b_0, \ldots, b_k)\in P_k$. Then, by the span condition, in a
	neighborhood $U$ of $a_0 = b_0'$,
	\[	
		F(a_0,\ldots,a_k) = \lambda F(b_0',\ldots, b_k')
	\]
	for some coefficient $\lambda$. Since by assumption,
	the left hand side
	does not vanish identically on $U$, also the right hand side does not
	vanish identically on $U$
	which exactly means $\ell(b_0',\ldots ,b_k') = \ell(b_0,\ldots,b_k) =
	0$.
\end{proof}

\begin{lem}\label{lem:ellone}
	Suppose that $F$ satisfies conditions (i)--(iv).

	If $\ell(a_0,\ldots,a_k)=1$ for some choice of points $(a_0 , \ldots ,
	a_k)\in P_k$, then $\ell \equiv 1$ on $P_k$.

	If $r(a_0,\ldots,a_k)=k-1$ for some choice of points $(a_0,
	\ldots,a_k)\in P_k$, then $r\equiv k-1$ on $P_k$.
\end{lem}
\begin{proof}
	We only prove the part about $\ell$, the proof for $r$ is similar.
	
	We begin by proving $k_\ell=1$. 
Let $a_0 < \cdots < a_k$ be a sequence of integers with
$\ell(a_0,\ldots,a_k)=1$ and $a_1=0$.
Inserting a point $\alpha$ to the right of $a_1$ yields two coefficients
$\lambda,\mu$ with
\[
	F(a_0,\ldots,a_k) = \lambda F(b_0,\ldots,b_k) + \mu
	F(b_1,\ldots,b_{k+1}),
\]
where $b_0 < \cdots < b_{k+1}$ is the increasing rearrangement of the points
$(a_0,\ldots,a_k,\alpha)$. Since by Lemma \ref{lem:ellzero},
$\ell(b_1,\ldots,b_{k+1})\geq 1$, we evaluate the above equation in a
sufficiently small neighborhood of $a_1$ which
yields $\lambda\neq 0$ and $\ell(b_0,\ldots, b_k) = 1$.
Performing this procedure iteratively and using dilation invariance of $F$, we
are able to choose an integer $n$ and construct point
sequences $a_0^{(j)} < \cdots < a_k^{(j)}$ for $j=1,\ldots $ satisfying
\begin{align*}
	a_i^{(j)} - a_{i-1}^{(j)} &= n^{j-i}\qquad \text{if }i \leq j, \\ 
	a_i^{(j)} - a_{i-1}^{(j)} &= 1\qquad \text{if }i>j
\end{align*}
for all $i=1,\ldots,k$ so that $\ell(a_0^{(j)},\ldots,a_k^{(j)}) = 1$ for all
$j$. Let $d$ be such that with $x:=\sum_{j=0}^d n^j$ we have $x/n \geq
k$. Next we choose $a_0 = 0, a_1 = x, a_{i+1} = a_i+ x/n$ for $i=1,\ldots,k-1$.
We refine by inserting the points 
\begin{align*}
	b_j &= \sum_{i= d-j+1}^d n^{i}, \qquad \text{for }j = 0,\ldots,d+1, \\
	b_j &= b_{j-1} + 1\qquad \text{for } j = d+2,\ldots.
\end{align*}
Observe that $b_{d+1} = a_1$ and $b_{d+2} < \cdots < b_{d+1+k} \leq a_2$ since
$x/n \geq k$.
Therefore, by the span condition and Lemma \ref{lem:ellzero} again, on $(-\infty,b_{d+2})$ we have
\begin{equation}\label{eq:geom}
	F(a_0,\ldots,a_k) = \sum_{j=0}^{d}  \lambda_j F(b_j,\ldots, b_{j+k}).
\end{equation}
Due to the above remarks and dilation invariance, $\ell(a_0,\ldots,a_k) =
\ell(b_j,\ldots,b_{j+k}) = 1$ for all $j=0,\ldots,d-1$ and
$\ell(b_j,\ldots,b_{j+k}) \geq 1$ for $j=d$.
Therefore, evaluating equation \eqref{eq:geom} successively on sufficiently
small neighborhoods of 
$b_1<\cdots<b_d$ respectively and using $b_d < a_1$,
we obtain $\lambda_0 = \cdots = \lambda_{d-1} = 0$ and we get
\begin{equation}\label{eq:geom1}
	F(a_0,\ldots,a_k) =  \lambda_d F(b_d,\ldots, b_{d+k}).
\end{equation}
Evaluating this equation on a sufficiently small neighborhood of $b_{d+1}=a_1$
we obtain $\ell(b_{d},\ldots,b_{d+k})
= 1$ but $b_{d} < \cdots < b_{d+k}$ is an equidistant partition, which
implies $k_\ell=1$.
	
Next, we show that for arbitrary $(a_0,\ldots,a_k)\in P_k$, we have
$\ell(a_0,\ldots,a_k)=1$. Without loss of generality, let $a_0< \cdots < a_k$ be integers. By
	Lemma \ref{lem:lreq}, since $k_\ell=1$, we know that $\ell(a_0,\ldots,a_k)=k_\ell=1$
	provided that $a_1 - a_0 = 1$. Let's now treat the case $a_1 - a_0 > 1$.
	
	Inductively we define a coarser set of points than $a_0 < \cdots  < a_k$
	that falls in the category where Lemma \ref{lem:lreq} can be applied and is suitable for our purpose.
	Set $b_0^{(0)} = a_0,\ldots ,b_k^{(0)} = a_k$.

	Assume that $b_0^{(j)} < \cdots < b_k^{(j)}$ is defined, then let $m\leq
	k$ be the smallest index with 
	\[
		b_m^{(j)} - b_{m-1}^{(j)} < a_1 - a_0.
	\]
	If no such index exists, we stop the induction and set $b_0 = b_0^{(j)},
	\ldots, b_k = b_k^{(j)}$. 
	If such an index exists, we define
	\begin{align*}
		b_i^{(j+1)} &= b_i^{(j)}&\text{if }i< m, \\
		b_i^{(j+1)} &= b_{i+1}^{(j)}&\text{if }m\leq i < k, \\
		b_k^{(j+1)} &= b_{k-1}^{(j+1)} + (a_1 - a_0).
	\end{align*}

	We remark that the procedure stops at some point and by Lemma
	\ref{lem:lreq} and dilation invariance, $\ell(b_0,\ldots,b_k)=1$. Now,
	let $a_{k+1} < \cdots < a_n$ be the points in $(b_0,\ldots,b_k)$ that
	are larger than $a_k$. Then the span condition implies
	\[
		F(b_0,\ldots,b_k) = \sum_{j=0}^{n-k} \lambda_j
		F(a_j,\ldots,a_{j+k})
	\]
	for some coefficients $(\lambda_j)$.

	If we now assume that $\ell(a_0,\ldots,a_k) > 1$ then the right hand
	side vanishes identically on $(-\infty,a_2)$ since $\ell(a_j, \ldots ,
	a_{j+k})\geq 1$ for all $j$ by Lemma \ref{lem:ellzero}. On the other
hand, the left hand side does not vanish identically on $(-\infty,a_2)$
since
	$a_2 > a_1 = b_1$ and $\ell(b_0,\ldots,b_k)=1$. Therefore, we conclude $\ell(a_0,\ldots,a_k)
	\leq 1$ and also, by Lemma \ref{lem:ellzero} again,
	$\ell(a_0,\ldots,a_k)= 1$.
\end{proof}

\begin{lem}\label{lem:elltwo}
	Suppose that $F$ satisfies conditions (i)--(iv).

	If $k_\ell = 2$ then $\ell\equiv 2$ on $P_k$.

	If $k_r = k-2$ then $r\equiv k-2$ on $P_k$.
\end{lem}
\begin{proof}
If $k_\ell=2$ then we also know that $\ell(a_0,\ldots,a_k) \geq 2$ for
all $(a_0,\ldots,a_k)\in P_k$
since otherwise Lemmas~\ref{lem:ellzero} and \ref{lem:ellone} would imply
$k_\ell < 2$.
Therefore, formula \eqref{eq:facts1} in the proof of Lemma~\ref{lem:lreq} can be
improved upon to
\[
	\left\{
		\begin{aligned}
			\ell(c_i,\ldots,c_{i+k}) & = 2\qquad &\text{if 
			$i<m$}, \\
			\ell(c_i,\ldots,c_{i+k}) &\geq 2\qquad&\text{ if
			$m\leq i$}, 
		\end{aligned}
		\right.
\]
Going then through the proof of Lemma \ref{lem:lreq} and using those improved
facts in the induction on $m$, we obtain that for any
$(a_0,\ldots,a_k) \in P_k$, we have $\ell(a_0,\ldots,a_k) = 2$.
Similar arguments work in the case $k-k_r =2$ yielding $r \equiv k-2$ on $P_k$
if $k_r = k-2$.
\end{proof}

Once we know that the functions $\ell$ and $r$ are constant, it is an easy
matter to deduce Theorem~\ref{thm:main_large} from
Theorem~\ref{thm:main_small} using Corollary~\ref{cor:lconstant}.
In fact, we can even show the following refinement of Theorem~\ref{thm:main_large}:
\begin{thm}\label{thm:main_large_refined}
	Suppose that $F$ satisfies conditions (i)-(iv) and that	one of
	the following two conditions is true:
	\begin{enumerate}
		\item there exists $(b_0,\ldots,b_k)\in P_k$ so that one of the four pairs 
	$\{b_0,b_k\},  \{b_0,b_{k-1}\}$, $ \{b_1,b_{k-1}\}, \{b_1,b_k\}$ is
	contained in the support of $F(b_0,\ldots,b_k)$.
\item $\max(k_\ell, k-k_r) \leq 2$.
	\end{enumerate}

	Then, for some non-negative integer $n$ and all $(a_0,\ldots,a_k) \in
	P_k$,
	$F(a_0,\ldots,a_k)$ is a constant, non-zero multiple of the $n$th distributional
	derivative of the B-spline function of order $k_r - k_\ell$ with respect
	to the points $a_{k_\ell}, a_{k_\ell+1} ,\ldots ,a_{k_r}$.
\end{thm}

\begin{proof}
	The first assumption (1) implies by
	Lemmas~\ref{lem:ellzero} and \ref{lem:ellone} that the functions $\ell$
	and $r$ are constant on $P_k$.
	The second assumption (2) implies by
	Lemmas~\ref{lem:elltwo}, \ref{lem:ellone}, \ref{lem:ellzero} that the 
	functions $\ell$ and $r$ are constant on
	$P_k$.
	
	In both cases, Corollary~\ref{cor:lconstant} then implies the assertion
	of Theorem \ref{thm:main_large}.
\end{proof}

The following theorem summarizes a few important consequences
of our results and it follows directly from Theorems~\ref{thm:main_small} and
\ref{thm:main_large_refined}.
\begin{thm}[B-Spline characterization]
	\label{thm:char}
	Let $F : P_k \to \mathscr E'(\mathbb R)$ be a mapping satisfying conditions (i)--(iv). Then,
	the following statements are true.
	\begin{enumerate}[(a)]
		\item If there exists $(a_0,\ldots,a_k)\in P_k$ so that
			$F(a_0,\ldots,a_k)$ is non-negative and if there exists
			$(b_0,\ldots,b_k)\in P_k$ so that $\{b_0,b_k\}$ is
			contained in the support of $F(b_0,\ldots,b_k)$ then,
			for all $(c_0,\ldots,c_k)\in P_k$, $F(c_0,\ldots,c_k)$ is
			a constant non-zero multiple of the B-spline
			corresponding to the points $(c_0,\ldots,c_k)$.
		\item If there exists $(a_0,\ldots,a_k)\in P_k$ so that
			$0\not\equiv F(a_0,\ldots,a_k)\in C^{k-2}$, then, for all
			$(b_0,\ldots,b_k)\in P_k$, $F(b_0,\ldots,b_k)$ is a
			constant, non-zero multiple of the B-spline
			corresponding to $(b_0,\ldots,b_k)$.
		\item There does not exist $(a_0,\ldots,a_k)\in P_k$ so that $0\not\equiv
			F(a_0,\ldots,a_k) \in C^{k-1}$.
	\end{enumerate}
\end{thm}

\appendix

\section{Proof of Proposition \ref{prop.non.uniq}}\label{sec:appendix}
\subsection{Some operations on $H$.}\label{ss.preparation}
For our construction of examples, we need two transformations of mappings $H$
satisfying (i)-(iv), or (i), (ii.e), (iii), (iv): increasing of dimension and
reflection (recall condition (ii.e) defined in Section~\ref{sss.v2}).
To describe them, fix $\kappa \geq 0$, and let $H$ be one of the following:
\begin{itemize}
\item[(a)] $H:E_\kappa \to \ccE'(\bR)$, satisfying (i), (ii.e), (iii), (iv).
\item[(b)] $H:P_\kappa \to \ccE'(\bR)$, satisfying (i)- (iv).
\end{itemize}
In particular, we are interested in the effect of these transformations to the
parameters
$$
k_{\ell}(H) = \min \supp H(0, \ldots, \kappa)
\quad \hbox{ and } \quad
k_r (H) = \max \supp H(0, \ldots, \kappa).
$$
\subsubsection{Reflection.}\label{sss.reflection}
Define
$$
J(a_0, \ldots, a_\kappa)(\cdot) = H(-a_\kappa, \ldots, -a_0)( - \cdot).
$$
Then:
\begin{itemize}
\item[(a.1)] In case $H$ is as in (a), then  $J:E_\kappa \to \ccE'(\bR)$ and satisfies (i), (ii.e), (iii), (iv).
\item[(b.1)] In case $H$ is as in (b), then  $J:P_\kappa \to \ccE'(\bR)$ and  satisfies (i)- (iv).
\end{itemize}
Moreover, in both cases
$$
k_{\ell}(J) = \kappa - k_r(H),
 \quad \quad
k_r(J) = \kappa - k_{\ell}(H).
$$
\subsubsection{Increasing of dimension.}\label{sss.dim}
Fix $k \geq \kappa$ and $\nu \in \bZ$ such that
$$
0 \leq \nu \leq \nu + \kappa \leq k.
$$
Define
$$
I(a_0, \ldots, a_k)(\cdot) = H(a_\nu, \ldots, a_{\nu +\kappa}).
$$
\begin{itemize}
\item[(a.2)] In case $H$ is as in (a), then  $I:E_k \to \ccE'(\bR)$ and satisfies (i), (ii.e), (iii), (iv).
\item[(b.2)] In case $H$ is as in (b), then  $I:P_k \to \ccE'(\bR)$ and  satisfies (i)- (iv).
\end{itemize}
Moreover, in both cases
$$
k_{\ell}(I) = \nu +  k_{\ell}(H),
 \quad \quad
k_r(I) = \nu +  k_{r}(H).
$$

\subsection{Extension of Example \ref{ex:counter}.}\label{ss.ext.counter}

Fix a subset $A \subset P_3$ satisfying the following conditions:
\begin{itemize}
\item[(c.1)] $A$ is translation and dilation invariant. 
\item[(c.2)]  Let $b_0<b_1 < b_2 < b_3  < b_4$. Then at most one of $(b_0, b_1, b_2, b_3)$, $(b_1, b_2, b_3, b_4)$ belongs to $A$.
\item[(c.3)] Splitting pattern for elements of $A$.

Let $(a_0, a_1,a_2,a_3) \in A$ and $a_0 < a < a_3$, $a \not \in \{a_1, a_2\}$. Let $\tilde{a}_0 < \tilde{a}_1 < \tilde{a}_2 < \tilde{a}_3 < \tilde{a}_4$
be an increasing rearrangement of $\{a_0, a_1,a_2,a_3, a\}$. Then:
\begin{itemize}
\item[(c.3.1)] If $a_0 < a < a_1$, then $(\tilde{a}_0 , \tilde{a}_1 , \tilde{a}_2 , \tilde{a}_3)  \not \in A$ and $( \tilde{a}_1 , \tilde{a}_2 , \tilde{a}_3, \tilde{a}_4)  \not \in A$. 
\item[(c.3.2)] If $a_1 < a < a_2$, then $(\tilde{a}_0 , \tilde{a}_1 , \tilde{a}_2 , \tilde{a}_3)  \not \in A$.
\item[(c.3.3)] If $a_2 < a < a_3$, then $(\tilde{a}_0 , \tilde{a}_1 , \tilde{a}_2 , \tilde{a}_3)   \in A$.

\end{itemize} 

\end{itemize}

We refer to Example \ref{ex:counter} for an explicit example of $A\subset P_3$
having those properties (note that in Example~\ref{ex:counter}, $A$ is defined
as a subset of $P_4$ but the conditions only involve the points
$(a_0,\ldots,a_3)$).

For $k \geq 4$ define 
$$
P_k (A) = \{ (a_0, \ldots, a_k) \in P_k: (a_0, a_1,a_2,a_3) \in A\}.
$$
Definition of  $P_k(A)$ and  conditions (c.1)-(c.3) for $A$ imply the following for $P_k(A)$:
\begin{itemize}
\item[(C.1)] $P_k(A)$ is translation and dilation invariant.
\item[(C.2)]  Let $b_0<\ldots < b_k < b_{k+1}$. Then at most one of $(b_0, \ldots, b_k)$, $(b_1, \ldots, b_{k+1})$ belongs to $P_k(A)$.
\item[(C.3)] Splitting pattern for elements of $P_k(A)$.

Let $(a_0,\ldots ,a_k) \in P_k( A)$ and $a_0 < a < a_k$, $a \neq a_i$, $i =1, \ldots, k-1$. Let $\tilde{a}_0 < \ldots < \tilde{a}_k < \tilde{a}_{k+1}$
be the increasing rearrangement of $\{a_0, \ldots ,a_k\} \cup \{ a\}$. Then:
\begin{itemize}
\item[(C.3.1)] If $a_0 < a < a_1$, then $(\tilde{a}_0 , \ldots , \tilde{a}_k)  \not \in P_k(A)$ and $( \tilde{a}_1 ,\ldots, \tilde{a}_{k+1})  \not \in P_k(A)$. 
\item[(C.3.2)] If $a_1 < a < a_2$, then $(\tilde{a}_0 , \ldots, \tilde{a}_k)  \not \in P_k(A)$.
\item[(C.3.3)] If $a_2 < a < a_k$, then $(\tilde{a}_0 , \ldots  , \tilde{a}_k)   \in P_k( A)$.

\end{itemize} 
\item[(C.4)]  Splitting pattern for elements of $P_k \setminus P_k(A)$. 

If $(a_0,\ldots ,a_k) \not\in P_k( A)$ and $a_3 < a < a_k$, then $(\tilde{a}_0 , \ldots  , \tilde{a}_k)  \not \in P_k( A)$.
\end{itemize}

Consider $F: P_k \to \ccE'(\bR)$ given by formula
\begin{align*}
	F(a_0, \ldots, a_k) &= N(a_2, a_4, \ldots, a_k)  \quad \text{if }  (a_0,
\ldots, a_k)  \in P_k(A), \\
F(a_0, \ldots, a_k) &= N(a_3, a_4, \ldots, a_k)  \quad \text{if }  (a_0,
\ldots, a_k)  \in P_k \setminus P_k(A).
\end{align*}
This family $F$ satisfies conditions (i)-(iv). Indeed, (i) is clear. Since
$P_k(A)$ is translation and dilation invariant,
the same is true for  $P_k\setminus P_k(A)$, hence (iii)-(iv) follow. 

It remains  to check the span condition (ii). 
For this, let $(a_0, \ldots, a_k) \in P_k$ and $a_0 < a < a_k$, $a \neq a_j$, $j =1, \ldots, k-1$, and let $\tilde{a}_0 < \tilde{a}_1 < \ldots < \tilde{a}_k< \tilde{a}_{k+1}$
be the increasing rearrangement of $\{a_0, a_1, \ldots, ,a_k, a\}$.
We need to check several cases.

 {\textsc{Case I}}: $(a_0, \ldots, a_k) \in P_k(A)$, so $F(a_0, \ldots, a_k) = N(a_2, a_4, \ldots, a_k)$.
We consider several sub-cases, according to the position of $a$ in the sequence $\tilde{a}_0 < \tilde{a}_1 < \ldots < \tilde{a}_k< \tilde{a}_{k+1}$.

\begin{itemize}
\item[(I.1)] $a_0<a<a_1$. Then $(\tilde{a}_0, \tilde{a}_1 ,\ldots , \tilde{a}_k) \not \in P_k(A)$ and 
$(\tilde{a}_1, \tilde{a}_2 ,\ldots , \tilde{a}_{k+1}) \not \in P_k(A)$, by (C.3.1). Therefore
\begin{eqnarray*}
F(\tilde{a}_0, \tilde{a}_1 ,\ldots , \tilde{a}_k) & = & N(\tilde{a}_3, \tilde{a}_4 ,\ldots , \tilde{a}_k) = N(a_2,  a_3, \ldots, a_{k-1}),
\\
F(\tilde{a}_1, \tilde{a}_2 ,\ldots , \tilde{a}_{k+1}) & = & N(\tilde{a}_4, \tilde{a}_5 ,\ldots , \tilde{a}_{k+1}) = N(a_3, a_4 , \ldots, a_{k}).
\end{eqnarray*}
Clearly, $N(a_2, a_4, \ldots, a_k) \in \lin\{ N(a_2,  a_3, \ldots, a_{k-1}), N(a_3, a_4,\ldots, a_{k})\}$.

\item[(I.2)] $a_1 < a < a_2$. Now, we have two possibilities.

If $(\tilde{a}_1, \tilde{a}_2 ,\ldots , \tilde{a}_{k+1}) \in P_k(A)$, then
$$
F(\tilde{a}_1, \tilde{a}_2 ,\ldots , \tilde{a}_{k+1}) = N(\tilde{a}_3, \tilde{a}_5, \ldots, \tilde{a}_{k+1}) = N(a_2, a_4, \ldots, a_k)
= F(a_0, \ldots, a_k) .
$$
Otherwise, we have $(\tilde{a}_1, \tilde{a}_2 ,\ldots , \tilde{a}_{k+1}) \not \in P_k(A)$, and $(\tilde{a}_0, \tilde{a}_1 ,\ldots , \tilde{a}_{k})  \not\in P_k(A)$
by (C.3.2). Therefore
\begin{eqnarray*}
F(\tilde{a}_0, \tilde{a}_1 ,\ldots , \tilde{a}_k) & = & N(\tilde{a}_3, \tilde{a}_4 ,\ldots , \tilde{a}_k) = N(a_2,  a_3, \ldots, a_{k-1}),
\\
F(\tilde{a}_1, \tilde{a}_2 ,\ldots , \tilde{a}_{k+1}) & = & N(\tilde{a}_4, \tilde{a}_5 ,\ldots , \tilde{a}_{k+1}) = N(a_3, a_4 , \ldots, a_{k}),
\end{eqnarray*}
and $N(a_2, a_4, \ldots, a_k) \in \lin\{ N(a_2,  a_3, \ldots, a_{k-1}), N(a_3, a_4,\ldots, a_{k})\}$.

\item[(I.3)] $a_2 < a < a_3$. Then $(\tilde{a}_0, \tilde{a}_1 ,\ldots , \tilde{a}_k) \in P_k(A)$, because of (C.3.3), and $(\tilde{a}_1, \tilde{a}_2 ,\ldots , \tilde{a}_{k+1}) \not \in P_k(A)$ because of (C.2). Consequently
\begin{eqnarray*}
F(\tilde{a}_0, \tilde{a}_1 ,\ldots , \tilde{a}_k) & = & N(\tilde{a}_2, \tilde{a}_4 ,\ldots , \tilde{a}_k) = N(a_2,  a_3, \ldots, a_{k-1}),
\\
F(\tilde{a}_1, \tilde{a}_2 ,\ldots , \tilde{a}_{k+1}) & = & N(\tilde{a}_4, \tilde{a}_5 ,\ldots , \tilde{a}_{k+1}) = N(a_3, a_4 , \ldots, a_{k}),
\end{eqnarray*}
and $N(a_2, a_4, \ldots, a_k) \in \lin\{ N(a_2,  a_3, \ldots, a_{k-1}), N(a_3, a_4,\ldots, a_{k})\}$. 

\item[(I.4)] $a_3 < a < a_k$.  Similarly as in case (I.3), $(\tilde{a}_0, \tilde{a}_1 ,\ldots , \tilde{a}_k) \in P_k(A)$, because of (C.3.3), and $(\tilde{a}_1, \tilde{a}_2 ,\ldots , \tilde{a}_{k+1}) \not \in P_k(A)$ because of (C.2). Therefore
\begin{eqnarray*}
F(\tilde{a}_0, \tilde{a}_1 ,\ldots , \tilde{a}_k) & = & N(\tilde{a}_2, \tilde{a}_4 ,\ldots , \tilde{a}_k) ,
\\
F(\tilde{a}_1, \tilde{a}_2 ,\ldots , \tilde{a}_{k+1}) & = & N(\tilde{a}_4, \tilde{a}_5 ,\ldots , \tilde{a}_{k+1}) .
\end{eqnarray*}
Note that  $\{  \tilde{a}_4, \ldots, \tilde{a}_k, \tilde{a}_{k+1}\} = \{a_4,
\ldots, a_k\} \cup \{a\}$, and consequently also we have
 $\{ \tilde{a}_2, \tilde{a}_4, \ldots, \tilde{a}_k, \tilde{a}_{k+1}\} = \{a_2, a_4, \ldots, a_k\} \cup \{a\}$. 
 Therefore  we have
$N(a_2, a_4, \ldots, a_k) \in \lin\{ N(\tilde{a}_2, \tilde{a}_4 ,\ldots , \tilde{a}_k), N(\tilde{a}_4, \tilde{a}_5 ,\ldots , \tilde{a}_{k+1}) \}$.
\end{itemize}

 {\textsc{Case II}: $(a_0, \ldots, a_k) \not\in P_k(A)$, so $F(a_0, \ldots, a_k) = N(a_3, a_4, \ldots, a_k)$.
Similarly as in case I, we consider some sub-cases, according to the position of $a$ in the sequence $\tilde{a}_0 < \tilde{a}_1 < \ldots < \tilde{a}_k< \tilde{a}_{k+1}$.
\begin{itemize}
\item[(II.1)] $a_0 < a < a_3$. 

If $ (\tilde{a}_1, \tilde{a}_2 ,\ldots , \tilde{a}_{k+1}) \not \in P_k(A)$, then
$$
F(\tilde{a}_1, \tilde{a}_2 ,\ldots , \tilde{a}_{k+1}) = N(\tilde{a}_4, \tilde{a}_5 ,\ldots , \tilde{a}_{k+1})
= N(a_3, a_4, \ldots, a_k) = F(a_0, \ldots, a_k) .$$

If $ (\tilde{a}_1, \tilde{a}_2 ,\ldots , \tilde{a}_{k+1})  \in P_k(A)$, then -- because of (C.2) --  $ (\tilde{a}_0, \tilde{a}_1 ,\ldots , \tilde{a}_{k}) \not \in P_k(A)$. Then
\begin{eqnarray*}
F(\tilde{a}_0, \tilde{a}_1 ,\ldots , \tilde{a}_k) & = & N(\tilde{a}_3, \tilde{a}_4 ,\ldots , \tilde{a}_{k}) ,
\\
F(\tilde{a}_1, \tilde{a}_2 ,\ldots , \tilde{a}_{k+1}) & = & N(\tilde{a}_3, \tilde{a}_5 ,\ldots , \tilde{a}_{k+1}) .
\end{eqnarray*}
Observe that $$\lin\{ N(\tilde{a}_3, \tilde{a}_4 ,\ldots , \tilde{a}_{k}) , N(\tilde{a}_3, \tilde{a}_5 ,\ldots , \tilde{a}_{k+1})  \} =
\lin\{ N(\tilde{a}_3, \tilde{a}_4 ,\ldots , \tilde{a}_{k}) , N(\tilde{a}_4, \tilde{a}_5 ,\ldots , \tilde{a}_{k+1}) \},$$
and in  case under consideration
$N(a_3, a_4, \ldots, a_k) =  N(\tilde{a}_4, \tilde{a}_5 ,\ldots ,
\tilde{a}_{k+1})$. So the span condition follows.

\item[(II.2)] $a_3 < a  <  a_k$. Then $(\tilde{a}_0, \tilde{a}_1 ,\ldots , \tilde{a}_k) \not\in P_k(A)$ by (C.4),
so 
$$
F(\tilde{a}_0, \tilde{a}_1 ,\ldots , \tilde{a}_k)  =  N(\tilde{a}_3, \tilde{a}_4 ,\ldots , \tilde{a}_{k}) .
$$
If $(\tilde{a}_1, \tilde{a}_2 ,\ldots , \tilde{a}_{k+1}) \not\in P_k(A)$, then
$$
F(\tilde{a}_1, \tilde{a}_2 ,\ldots , \tilde{a}_{k+1})  =  N(\tilde{a}_4, \tilde{a}_5 ,\ldots , \tilde{a}_{k+1}) .
$$
If $(\tilde{a}_1, \tilde{a}_2 ,\ldots , \tilde{a}_{k+1}) \in P_k(A)$, then
$$
F(\tilde{a}_1, \tilde{a}_2 ,\ldots , \tilde{a}_{k+1})  =  N(\tilde{a}_3, \tilde{a}_5 ,\ldots , \tilde{a}_{k+1}) .
$$
In both cases
$$
\lin\{  F(\tilde{a}_0, \tilde{a}_1 ,\ldots , \tilde{a}_k) , F(\tilde{a}_1, \tilde{a}_2 ,\ldots , \tilde{a}_{k+1})  \}
= \lin\{ N(\tilde{a}_3, \tilde{a}_4 ,\ldots , \tilde{a}_{k}), N(\tilde{a}_4, \tilde{a}_5 ,\ldots , \tilde{a}_{k+1}) \}.
$$
Note that $\{\tilde{a}_3, \tilde{a}_4 ,\ldots , \tilde{a}_{k}, \tilde{a}_{k+1}  \} = \{a_3, \ldots, a_k\} \cup \{a\}$. It follows that
\begin{eqnarray*}
F(a_0, \ldots, a_k)  & = & N(a_3, a_4, \ldots, a_k) 
\\
& \in &  \lin\{ N(\tilde{a}_3, \tilde{a}_4 ,\ldots , \tilde{a}_{k}), N(\tilde{a}_4, \tilde{a}_5 ,\ldots , \tilde{a}_{k+1}) \}
\\ & = & \lin\{  F(\tilde{a}_0, \tilde{a}_1 ,\ldots , \tilde{a}_k) , F(\tilde{a}_1, \tilde{a}_2 ,\ldots , \tilde{a}_{k+1})  \}.
\end{eqnarray*}
\end{itemize}

\subsection{The proof of Proposition \ref{prop.non.uniq}.} \label{ss.proof}
The example discussed in section \ref{ss.ext.counter} implies that $F$ has more that one extension to $P_k$ in case of $m=0$, $k \geq 4$, 
$k_\ell =3$ and $k_r = k$. It can be extended to $3 \leq k_\ell < k_r \leq k$ by increasing of dimension procedure, described in section \ref{sss.dim}.
Differentiating this example $m$ times we get more than one extension to $P_k$ in case of $m \geq 0$, $k \geq 4$ and $3 \leq k_\ell < k_r \leq k$.
Finally, the reflection procedure, described in section \ref{sss.reflection} allows the change of roles of $k_\ell$ and $k_r$.

\subsection*{Acknowledgements } 
M. Passenbrunner is  supported by the Austrian Science Fund FWF, project P32342.

\bibliographystyle{plain}
\bibliography{char}

\begin{thebibliography}{1}

\bibitem{refinable_anal}
Xin-Rong Dai.
\newblock Compactly supported multi-refinable distributions and {B}-splines.
\newblock {\em J. Math. Anal. Appl.}, 323(1):379--386, 2006.

\bibitem{refinable_east}
Xinrong Dai, Qiyu Sun, and Zeyin Zhang.
\newblock Compactly supported both {$m$} and {$n$} refinable distributions.
\newblock {\em East J. Approx.}, 6(2):201--209, 2000.

\bibitem{uncond_franklin}
Gegham~G. Gevorkyan and Anna Kamont.
\newblock Unconditionality of general {F}ranklin systems in {$L^p[0,1],
  1<p<\infty$}.
\newblock {\em Studia Math.}, 164(2):161--204, 2004.

\bibitem{uncond}
M.~Passenbrunner.
\newblock Unconditionality of orthogonal spline systems in {$L^p$}.
\newblock {\em Studia Math.}, 222(1):51--86, 2014.

\bibitem{ae}
M.~Passenbrunner and A.~Shadrin.
\newblock On almost everywhere convergence of orthogonal spline projections
  with arbitrary knots.
\newblock {\em J. Approx. Theory}, 180:77--89, 2014.

\bibitem{Schumaker2007}
L.~L. Schumaker.
\newblock {\em Spline functions: basic theory}.
\newblock Cambridge Mathematical Library. Cambridge University Press,
  Cambridge, {T}hird edition, 2007.

\bibitem{shadrin}
A.~Yu. Shadrin.
\newblock The {$L_\infty$}-norm of the {$L_2$}-spline projector is bounded
  independently of the knot sequence: a proof of de {B}oor's conjecture.
\newblock {\em Acta Math.}, 187(1):59--137, 2001.

\bibitem{refinable_approx}
Qiyu Sun and Zeyin Zhang.
\newblock A characterization of compactly supported both {$m$} and {$n$}
  refinable distributions.
\newblock {\em J. Approx. Theory}, 99(1):198--216, 1999.

\end{thebibliography}

\end{document}